\documentclass[12pt,a4paper]{amsart}
\usepackage[latin1]{inputenc}   
\usepackage{ dsfont }
\usepackage{amsmath}
\usepackage{amsfonts}
\usepackage{amssymb}
\usepackage{esint}
\usepackage{graphicx}
\usepackage{amsthm}
\usepackage{enumerate}
\usepackage{verbatim}
\usepackage{hyperref}
\usepackage[english]{babel}
\usepackage{bbm}
\usepackage{mathrsfs}
\usepackage{relsize}
\usepackage{textcomp}
\usepackage[leqno]{amsmath}
\usepackage{color}
\makeatletter
\newcommand{\leqnomode}{\tagsleft@true}
\newcommand{\reqnomode}{\tagsleft@false}
\makeatother

\def \N{\mathbb{N}}
\def \R{\mathbb{R}}

\theoremstyle{plain} 
\newtheorem{thm}{Theorem}[section] 
\newtheorem{cor}[thm]{Corollary} 
\newtheorem{lem}[thm]{Lemma} 
\newtheorem{prop}[thm]{Proposition} 
\newtheorem{hp}[thm]{Hypotheses} 
\newtheorem{defn}[thm]{Definition}
 
\theoremstyle{definition} 
\newtheorem{rem}[thm]{Remark}
\numberwithin{equation}{section}

\newcommand{\miezz}{\frac{1}{2}}

\newcommand{\eps}{\varepsilon}

\newcommand{\into}{\ensuremath{\int_{\Omega}}}

\newcommand{\inti}{\ensuremath{\int_{0}^{t}\int_{\Omega}}}

\newcommand{\intif}{\ensuremath{\int_{0}^{T}\int_{\Omega}}}

\newcommand{\intc}[1]{\ensuremath{\int_{#1}^{T}\int_{\Omega}}}
\newcommand{\intm}[1]{\ensuremath{\int_{0}^{#1}\int_{\Omega}}}
\newcommand{\norm}[1]{\ensuremath{\left\Arrowvert #1 \right\Arrowvert}}
\newcommand{\norminf}[1]{\ensuremath{\left\Arrowvert #1 \right\Arrowvert_\infty}}

\newcommand{\spazio}{\hspace{0.08cm}}

\newcommand{\dm}[1]{\ensuremath{\frac{\delta #1}{\delta m}}}
\newcommand{\dw}{\mathbf{d}_1}

\newcommand{\supo}{\sup\limits_{t\in[0,T]}}
\newcommand{\tr}[1]{\mathrm{tr}(a(x)D^2#1)}
\newcommand{\rt}{\tilde{\rho}}

\newcommand{\bdone}[1]{a(x)D#1\cdot\nu_{|\partial\Omega}=0}
\newcommand{\am}{_{\frac{\alpha}{2},\alpha}}
\newcommand{\amu}{_{\frac{1+\alpha}{2},1+\alpha}}
\newcommand{\amd}{_{1+\frac{\alpha}{2},2+\alpha}}
\newcommand{\amv}{_{1,2+\alpha}}
\newcommand{\daw}[1]{\mathbf{d}_{#1}}

\newcommand{\be}{\begin{equation}}
	\newcommand{\ee}{\end{equation}}
\newcommand{\amf}{_{L^1(W^{-1,\infty})}}
\newcommand{\amc}{_{-(1+\alpha),N}}
\newcommand{\amn}{_{\gamma,-(1+\alpha),N}}
\newcommand{\amb}{_{-(1+\alpha)}}
\newcommand{\amp}{_{L^q(W^{-1,\infty})}}
\newcommand{\amw}{_{W^{-1,\infty}}}
\begin{document}
	
	\title{The Master Equation in a Bounded Domain with Neumann conditions}
	
	\author{Michele Ricciardi}\thanks{Dipartimento di Informatica, Universit\`{a} degli studi di Verona.
		Via S. Francesco, 22, 37129 Verona (VR), Italy.
		\texttt{michele.ricciardi@univr.it}}

	\date{\today}

	\maketitle
	
	\begin{abstract}
		In this article we study the well-posedness of the Master Equation of Mean Field Games in a framework of Neumann boundary condition. The definition of solution is closely related to the classical one of the Mean Field Games system, but the boundary condition here leads to two Neumann conditions in the Master Equation formulation, for both space and measure. The global regularity of the linearized system, which is crucial in order to prove the existence of solutions, is obtained with a deep study of the boundary conditions and the global regularity at the boundary of a suitable class of parabolic equations.
	\end{abstract}

	\section{Introduction}
	
	Mean Field Games theory is devoted to the study of differential games with a large number $N$ of small and indistinguishable agents. The theory was initially introduced by J.-M. Lasry and P.-L. Lions in 2006 (\cite{LL1, LL2, LL3, LL-japan}), using tools from mean-field theories, and in the same years by P. Caines, M. Huang and R. Malham\'{e} \cite{HCM}.
	
	The macroscopic description used in mean field game theory leads to study coupled systems of PDEs, where the Hamilton-Jacobi-Bellman equation satisfied by the single agent's value function $u$ is coupled with the Kolmogorov Fokker-Planck equation satisfied by the distribution law of the population $m$. The simplest form of this system is the following
	\begin{equation}\label{meanfieldgames}
		\begin{cases}
			-\partial_t u - \mathrm{tr}(a(x)D^2u) +H(x,Du)=F(x,m)\,,\\
			\partial_t m - \sum\limits_{i,j} \partial_{ij}^2 (a_{ij}(x)m) -\mathrm{div}(mH_p(x,Du))=0\,,\\
			m(0)=m_0\,, \hspace{2cm} u(T)=G(x,m(T))\,.
		\end{cases}
	\end{equation}
	Here, $H$ is called the \emph{Hamiltonian} of the system, whereas $a$ is a uniformly elliptic matrix, representing the square of the diffusion term in the stochastic dynamic of the generic player, and $F$ and $G$ are the running cost and the final cost related to the generic player.\\
	
	In his lectures at Coll\`{e}ge de France \cite{prontoprontopronto}, P.-L. Lions proved that the solutions $(u,m)$ of \eqref{meanfieldgames} are the trajectories of a new infinite dimensional partial differential equation.
	This $PDE$ is called \textbf{\emph{Master Equation}} and summarizes the informations contained in \eqref{meanfieldgames} in a unique equation.
	
	The definition of the Master Equation is related to its characteristics, which are solution of the MFG system. To be more precise, considering the solution $(u,m)$ of the system \eqref{meanfieldgames} with initial condition $m(t_0)=m_0$, one defines the function
	\begin{equation}\label{defu}
		U:[0,T]\times\Omega\times\mathcal{P}(\Omega)\to\R\,,\qquad U(t_0,x,m_0)=u(t_0,x)\,,
	\end{equation}
	where $\Omega\subseteq\R^d$ and $\mathcal{P}(\Omega)$ is the set of Borel probability measures on $\Omega$.
	
	In order to give sense at this definition, the equation \eqref{meanfieldgames} must have a unique solution defined in $[0,T]\times\Omega$, for all $(t_0,m_0)$. So, we assume that $F$ and $G$ are \emph{monotone} functions with respect to the measure variable, a structure condition which ensures the existence of a unique solution for large time interval. 
	
	If we compute, at least formally, the equation satisfied by $U$, we obtain a Hamilton-Jacobi equation in the space of measures, called Master Equation.\\
	
	The relevance of the Master Equation was recognized in different papers, and most important topics like existence, uniqueness and regularity results have been developed. For example, in \cite{14} and \cite{15} Bensoussan, Frehse and Yam reformulated this equation as a PDE set on an $L^2$ space, and in \cite{24} Carmona and Delarue interpreted it as a decoupling field of forward-backward stochastic differential equations in infinite dimension. A very general result of well-posedness of the Master Equation was given by Cardaliaguet, Delarue, Lasry and Lions in \cite{card}.
	
	So far, most of the literature, especially in the Master Equation's papers, considers the case where the state variable $x$ belongs to the flat torus (i.e. periodic solutions, $\Omega=\mathbb{T}^d$), or, especially in the probabilistic literature, in the whole space $\R^d$.
	
	But in many economic and financial applications it is useful to work with a process that remains in a certain domain of existence; thus, some conditions at the boundary need to be prescribed. See for instance the models analyzed by Achdou, Buera et al. in \cite{gol}.
	
	In this paper we want to analyze this situation, by studying the well-posedness of the Master Equation in a framework of Neumann condition at the boundary.
	
	In this case the MFG system \eqref{meanfieldgames} is constrained with the following boundary conditions: for all $x\in\partial\Omega$
	\begin{equation}\label{fame}
		a(x)D_x u(t,x)\cdot\nu(x)=0\,,\qquad \mathlarger{[}a(x) Dm(t,x)+m(H_p(x,Du)+\tilde{b}(x))\mathlarger{]}\cdot\nu(x)=0\,,
	\end{equation}
	where $\nu(\cdot)$ is the outward normal at $\partial\Omega$ and $\tilde{b}$ is a vector field defined as follows:
	$$
	\tilde{b}_i(x)=\mathlarger{\sum}\limits_{j=1}^d\frac{\partial a_{ji}}{\partial x_j}(x)\spazio,\hspace{2cm}i=1,\dots,d\spazio.
	$$\,.\\
	
	The Master Equation, in this case, takes the following form
	\begin{equation}\begin{split}
			\label{Master}
			\left\{
			\begin{array}{rl}
				&-\,\partial_t U(t,x,m)-\mathrm{tr}\left(a(x)D_x^2 U(t,x,m)\right)+H\left(x,D_x U(t,x,m)\right)\\&-\mathlarger{\into}\mathrm{tr}\left(a(y)D_y D_m U(t,x,m,y)\right)dm(y)\\&+\mathlarger{\into} D_m U(t,x,m,y)\cdot H_p(y,D_x U(t,y,m))dm(y)= F(x,m)\\&\mbox{in }(0,T)\times\Omega\times\mathcal{P}(\Omega)\spazio,\vspace{0.4cm}\\
				&U(T,x,m)=G(x,m)\hspace{1cm}\mbox{in }\Omega\times\mathcal{P}(\Omega)\spazio,\vspace{0.2cm}\\
				&a(x)D_x U(t,x,m)\cdot\nu(x)=0\hspace{1cm}\quad\,\mbox{for }(t,x,m)\in(0,T)\times\partial\Omega\times\mathcal{P}(\Omega)\,,\\
				&a(y)D_m U(t,x,m,y)\cdot\nu(y)=0\hspace{1cm}\mbox{for }(t,x,m,y)\in(0,T)\times\Omega\times\mathcal{P}(\Omega)\times\partial\Omega\,,
			\end{array}
			\right.
	\end{split}\end{equation}
	where $D_mU$ is a derivation of $U$ with respect to the measure, whose precise definition will be given later. This definition, anyway, is strictly related to the one given by Ambrosio, Gigli and Savar\'{e} in \cite{ags} and by Lions in \cite{prontoprontopronto}.\\
	
	We stress the fact that the last boundary condition, i.e.
	$$
	a(y)D_m U(t,x,m,y)\cdot\nu(y)=0\hspace{1cm}\mbox{for }(t,x,m,y)\in(0,T)\times\Omega\times\mathcal{P}(\Omega)\times\partial\Omega\,,
	$$
	
	is completely new in the literature. It relies on the fact that  we have to pay attention to the space where $U$ is defined, i.e. $[0,T]\times\Omega\times\mathcal{P}(\Omega)$. Then, together with final data and Neumann condition with respect to $x$, there is another boundary condition caused by the boundary of $\mathcal{P}(\Omega)$.\\
	
	The importance of this equation arises in the so-called \emph{convergence problem}. The Mean Field Games system approximates the $N$-player differential game, in the sense that the optimal strategies in the Mean Field Games system provide approximated Nash equilibria (called $\eps$-Nash equilibria) in the $N$-player game. See, for instance, \cite{resultuno}, \cite{resultdue}, \cite{resulttre}.
	
	Conversely, the convergence of the Nash Equilibria in the $N$-player game towards an optimal strategy in the Mean Field Games presents many difficulties, due to the lack of compactness properties of the problem. Hence, the Master Equation plays an instrumental role in order to study this problem. A convergence result in a framework of Neumann conditions at the boundary will be given in the forthcoming paper \cite{prossimamente}.\\
	
	There are many papers about the well-posedness of the Master Equation. We point out here that these papers are studied in two different contexts: the first case is the so-called \emph{First order Master Equation}, studied in this article, where the Brownian motions in the dynamic of the agents in the $N$-player differential game are independent each other. The second case is the \emph{Second order Master Equation}, or \emph{Master Equation with common noise}. In this case, the dynamic has also 
	an additional Brownian term $dW_t$, which is common to all players. This leads to a different and more difficult type of Master Equation, with some additional terms depending also on the second derivative $D_{mm}U$. It is relevant to say that Mean Field Games with common noise were already studied by Carmona, Delarue and Lacker in \cite{loacker}.
	
	Some preliminary results about the Master Equation were given by Lions in \cite{prontoprontopronto} and a first exhaustive result of existence and uniqueness of solutions was proved, with a probabilistic approach, by Chassagneux, Crisan and Delarue in \cite{28}, who worked in a framework with diffusion and without common noise. Buckhdan, Li, Peng and Rainer in \cite{bucchin} proved the existence of a classical solution using probabilistic arguments, when there is no coupling and no common noise. Furthermore, Gangbo and Swiech proved a short time existence for the Master Equation with common noise, see \cite{nuova16}.
	
	But the most important result in this framework was achieved by Cardaliaguet, Delarue, Lasry and Lions in \cite{card}, who proved existence and uniqueness of solutions for the Master Equation with and without common noise, including applications to the convergence problem, in a periodic setting ($\Omega=\mathbb{T}^d$). Other recent results about Master Equation and convergence problem can be found in \cite{dybala, nuova1, gomez, nuova14, cicciocaputo, nuova11, ramadan, nuova4, fifa21, tonali, koulibaly}.\\
	
	This article follows the main ideas of \cite{card}, but many issues appear, connected to the Neumann boundary condition, and more effort has to be done in order to gain the same results.
	
	The function $U$ is defined as in \eqref{defu} and some estimates like global bounds and global Lipschitz regularity are proved.
	The main issue in order to obtain that $U$ solves \eqref{Master} is to prove the $\mathcal{C}^1$ character of $U$ with respect to $m$. This step requires a careful analysis of the linearized mean field game system (see \cite{card}) in order to prove strong regularity of U in the space and in the measure variable.
	
	However, these estimates requires strong regularities of $U$, and so of the Mean Field Games system, in the space and in the measure variable.
	
	The space regularity is obtained in \cite{card} by differentiating the equation with respect to $x$. But in the Neumann case, and in general in any case of boundary conditions, these methods obviously cannot be applied so straightly, and these bounds are obtained using different kind of space-time estimates, which must be handled with care.
	
	Indeed, regularity estimates for Neumann parabolic equation require compatibility conditions between initial and boundary data. Unfortunately, these compatibility conditions will be not always guaranteed in this context. This forces us to generalize the estimates obtained in \cite{card}, by a deeper study of the regularity of solutions for the Fokker-Planck equation.\\ 
	
	The article is divided as follows.
	
	In section $2$ we define some useful tools and we state the main assumptions we will need in order to prove the next results.
	
	In the rest of the article (section $3$ to $6$), we analyze the well-posedness of the Master Equation \eqref{Master}.
	
	The idea is quite classical: for each $(t_0,m_0)$ we consider the $MFG$ system \eqref{meanfieldgames} in $[t_0,T]\times\Omega$, with conditions \eqref{fame}, and we define
	\begin{equation}\label{U}
		U(t_0,x,m_0)=u(t_0,x)\,.
	\end{equation}
	Then we prove that $U$ is a solution of the Master Equation.\\
	
	Section $3-5$ are completely devoted to prove technical results to ensure this kind of differentability.
	
	In section $3$ we prove a first estimate of a solution $(u,m)$ of the Mean Field Games system, namely
	$$
	\norm{u}\amd\le C\,,\qquad\dw(m(t),m(s))\le C|t-s|^\miezz\,,
	$$
	where $\dw$ is  a distance between measures called \emph{Wasserstein distance}, whose definition will be given in Section $3.2$. 
	
	In section $4$ we use the definition of $U$ from the Mean Field Games system in order to prove a Lipschitz character of $U$ with respect to $m$:
	$$
	\norm{U(t,\cdot,m_1)-U(t,\cdot,m_2)}_{2+\alpha}\le C\dw(m_1,m_2)\,.
	$$
	
	In section $5$ we prove the $\mathcal{C}^1$ character of $U$ with respect to $m$. This goes through different estimates on linearized MFG systems.
	
	Once proved the $\mathcal{C}^1$ character of $U$, we can prove that $U$ is actually the unique solution of the Master Equation \eqref{Master}. This will be done in Section $6$.\\

	\section{Notation, Assumptions and Main Result}
	
	Throughout this chapter, we fix a time $T>0$. $\Omega\subset\R^d$ will be the closure of an open bounded set, with boundary of class $\mathcal{C}^{2+\alpha}$, and we define $Q_T:=[0,T]\times\Omega$.
	
	For $n\ge0$ and $\alpha\in(0,1)$ we denote with $\mathcal{C}^{n+\alpha}(\Omega)$, or simply $\mathcal{C}^{n+\alpha}$, the space of functions $\phi\in\mathcal{C}^n(\Omega)$ with, for each $\ell\in\N^r$, $1\le r\le n$, the derivative $D^\ell\phi$ is H\"{o}lder continuous with H\"{o}lder constant $\alpha$. The norm is defined in the following way:
	\begin{align*}
		\norm{\phi}_{n+\alpha}:=\sum\limits_{|\ell|\le n}\norminf{D^l\phi}+\sum\limits_{|\ell|= n}\sup\limits_{x\neq y}\frac{|D^\ell\phi(x)-D^\ell\phi(y)|}{|x-y|^\alpha}\,.
	\end{align*}
	
	Sometimes, in order to deal with Neumann boundary conditions, we will need to work with a suitable subspace of $\mathcal{C}^{n+\alpha}(\Omega)$.
	
	So we will call $\mathcal{C}^{n+\alpha,N}(\Omega)$, or simply $\mathcal{C}^{n+\alpha,N}$, the set of functions $\phi\in\mathcal{C}^{n+\alpha}$ such that $aD\phi\cdot\nu_{|\partial\Omega}=0$, endowed with the same norm $\norm{\phi}_{n+\alpha}$.\\
	
	Then, we define several parabolic spaces we will need to work with during the chapter, starting from $\mathcal{C}^{\frac{n+\alpha}{2},n+\alpha}([0,T]\times\Omega)$.
	
	We say that $\phi:[0,T]\times\Omega\to\R$ is in $\mathcal{C}^{\frac{n+\alpha}{2},n+\alpha}([0,T]\times\Omega)$ if $\phi$ is continuous in both variables, together with all derivatives $D_t^rD_x^s\phi$, with $2r+s\le n$. Moreover, $\norm{\phi}_{\frac{n+\alpha}{2},n+\alpha}$ is bounded, where
	\begin{align*}
		\norm{\phi}_{\frac{n+\alpha}{2},n+\alpha}:=\sum\limits_{2r+s\le n}\norminf{D_t^rD^s_x\phi}&+\sum\limits_{2r+s=n}\sup\limits_t\norm{D_t^rD_x^s\phi(t,\cdot)}_{\alpha}\\&+\sum\limits_{0<n+\alpha-2r-s<2}\sup\limits_x\norm{D_t^rD_x^s\phi(\cdot,x)}_{\frac{n+\alpha-2r-s}{2}}\,.
	\end{align*}
	
	The space of continuous space-time functions which satisfy a H\"{o}lder condition in $x$ will be denoted by $\mathcal{C}^{0,\alpha}([0,T]\times\Omega)$. It is endowed with the norm
	$$
	\norm{\phi}_{0,\alpha}=\sup\limits_{t\in[0,T]}\norm{\phi(t,\cdot)}_\alpha\,.
	$$
	The same definition can be given for the space $\mathcal{C}^{\alpha,0}$. Finally, we define the space $\mathcal{C}^{1,2+\alpha}$ of functions differentiable in time and twice differentiable in space, with all derivatives in $\mathcal{C}^{0,\alpha}(\overline{Q_T})$. The natural norm for this space is
	$$
	\norm{\phi}_{1,2+\alpha}:=\norminf{\phi}+\norm{\phi_t}_{0,\alpha}+\norminf{D_x\phi}+\norm{D^2_x\phi}_{0,\alpha}\,.
	$$
	
	We note that, thanks to \emph{Lemma 5.1.1} of \cite{lunardi}, the first order derivatives of $\phi\in\mathcal{C}^{1,2}$ satisfy also a H\"{o}lder condition in time. Namely
	\begin{equation}\label{precisissimongulaeva}
		\norm{D_x\phi}_{\miezz,\alpha}\le C\norm{f}_{1,2+\alpha}\,.
	\end{equation}
	
	In order to study distributional solutions for the Fokker-Planck equation, we also need to define a structure for the dual spaces of regular functions.
	
	We define, for $n\ge0$ and $\alpha\in(0,1)$, the space $\mathcal{C}^{-(n+\alpha)}(\Omega)$, called for simplicity $\mathcal{C}^{-(n+\alpha)}$ in this article, as the dual space of $\mathcal{C}^{n+\alpha}$, endowed with the norm
	$$
	\norm{\rho}_{-(n+\alpha)}=\sup\limits_{\norm{\phi}_{n+\alpha}\le 1}\langle\rho,\phi\rangle\,.
	$$
	
	With the same notations we define the space $\mathcal{C}^{-(n+\alpha),N}$ as the dual space of $\mathcal{C}^{n+\alpha,N}$ endowed with the same norm:
	$$
	\norm{\rho}_{-(n+\alpha),N}=\sup\limits_{\substack{\norm{\phi}_{n+\alpha}\le 1\\aD\phi\cdot\nu_{|\partial\Omega}=0}}\langle \rho,\phi \rangle\,.
	$$
	
	Finally, for $k\ge1$ and $1\le p\le+\infty$, we can also define the space $W^{-k,p}(\Omega)$, called for simplicity $W^{-k,p}$, as the dual space of $W^{k,p}(\Omega)$, endowed with the norm
	$$
	\norm{\rho}_{W^{-k,p}}=\sup\limits_{\norm{\phi}_{W^{k,p}}\le 1}\langle\rho,\phi\rangle\,.
	$$
	
	\begin{defn}
		Let $m_1,m_2\in\mathcal{P}(\Omega)$ two Borel probability measures on $\Omega$.\\
		We call the \emph{Wasserstein distance} between $m_1$ and $m_2$, and we write $\dw(m_1,m_2)$ the quantity
		\begin{equation}\label{wass1}
			\dw(m_1,m_2):=\sup\limits_{Lip(\phi)\le 1}\into \phi(x)d(m_1-m_2)(x)\,.
		\end{equation}
	\end{defn}
	We note that we can also write \eqref{wass1} as
	\begin{align}\label{wass}
		\dw(m_1,m_2):=\sup\limits_{\substack{\norm{\phi}_{W^{1,\infty}}\le C\\Lip(\phi)\le 1}}\into \phi(x)d(m_1-m_2)(x)\,,
	\end{align}
	for a certain $C>0$. Actually, for a fixed $x_0\in\Omega$, we can restrict ourselves to the functions $\phi$ such that $\phi(x_0)=0$, since
	$$
	\into \phi(x)d(m_1-m_2)(x)=\into (\phi(x)-\phi(x_0))d(m_1-m_2)(x)\,,
	$$
	and these functions obviously satisfies $\norm{\phi}_{W^{1,\infty}}\le C$ for a certain $C>0$.
	
	We will always work with \eqref{wass}, where the restriction in $W^{1,\infty}$ allows us to obtain some desired estimates with respect to $\dw$.

	In order to give a sense to equation \eqref{Master}, we need to define a suitable derivation of $U$ with respect to the measure $m$.
	
	\begin{defn}\label{dmu}
		Let $U:\mathcal{P}(\Omega)\to\R$. We say that $U$ is of class $\mathcal{C}^1$ if there exists a continuous map $K:\mathcal{P}(\Omega)\times\Omega\to\R$ such that, for all $m_1$, $m_2\in\mathcal{P}(\Omega)$ we have
		\begin{equation}\label{deu}
			\lim\limits_{t\to0}\frac{U(m_1+s(m_2-m_1))-U(m_1)}{s}=\into K(m_1,x)(m_2(dx)-m_1(dx))\,.
		\end{equation}
	\end{defn}
	
	Note that the definition of $K$ is up to additive constants. Then, we define the derivative $\dm{U}$ as the unique map $K$ satisfying \eqref{deu} and the normalization convention
	$$
	\into K(m,x)dm(x)=0\,.
	$$
	
	As an immediate consequence, we obtain the following equality, that we will use very often in the rest of the chapter: for each $m_1$, $m_2\in\mathcal{P}(\Omega)$ we have
	$$
	U(m_2)-U(m_1)=\int_0^1\into\dm{U}((m_1)+s(m_2-m_1),x)(m_2(dx)-m_1(dx))\,.
	$$
	
	Finally, we can define the \emph{intrinsic derivative} of $U$ with respect to $m$.
	\begin{defn}\label{Dmu}
		Let $U:\mathcal{P}(\Omega)\to\R$. If $U$ is of class $\mathcal{C}^1$ and $\dm{U}$ is of class $\mathcal{C}^1$ with respect to the last variable, we define the intrinsic derivative $D_mU:\mathcal{P}(\Omega)\times\Omega\to\R^d$ as
		$$
		D_mU(m,x):=D_x\dm{U}(m,x)\,.
		$$
	\end{defn}

	We need the following assumptions:
	\begin{hp}\label{ipotesi}
		Assume that
		\begin{itemize}
			\item [(i)] (Uniform ellipticity) $\norm{a(\cdot)}_{1+\alpha}<\infty$ and $\exists\,\mu>\lambda>0$ s.t. $\forall\xi\in\mathbb{R}^d$ $$ \mu|\xi|^2\ge\langle a(x)\xi,\xi\rangle\ge\lambda|\xi|^2\,;$$
			\item [(ii)]$H:\Omega\times\R^d\to\R$, $G:\Omega\times\mathcal{P}(\Omega)\to\R$ and $F:\Omega\times\mathcal{P}(\Omega)\to\R$ are smooth functions with $H$ Lipschitz with respect to the last variable;
			\item [(iii)]$\exists C>0$ s.t.
			$$
			0< H_{pp}(x,p)\le C I_{d\times d}\spazio;
			$$
			\item [(iv)]$F$ satisfies, for some $0<\alpha<1$ and $C_F>0$,
			$$
			\into \left(F(x,m)-F(x,m')\right) d(m-m')(x)\ge0
			$$
			and
			$$
			\sup\limits_{m\in\mathcal{P}(\Omega)}\left(\norm{F(\cdot,m)}_{\alpha}+\norm{\frac{\delta F}{\delta m}(\cdot,m,\cdot)}_{\alpha,2+\alpha}\right)+\mathrm{Lip}\left(\dm{F}\right)\le C_F\spazio,
			$$
			with
			$$
			\mathrm{Lip}\left(\dm{F}\right):=\sup\limits_{m_1\neq m_2}\left(\dw(m_1,m_2)^{-1}\norm{\dm{F}(\cdot,m_1,\cdot)-\dm{F}(\cdot,m_2,\cdot)}_{\alpha,1+\alpha}\right)\spazio;
			$$
			\item [(v)]$G$ satisfies the same estimates as $F$ with $\alpha$ and $1+\alpha$ replaced by $2+\alpha$, i.e.
			$$
			\sup\limits_{m\in\mathcal{P}(\Omega)}\left(\norm{G(\cdot,m)}_{2+\alpha}+\norm{\frac{\delta G}{\delta m}(\cdot,m,\cdot)}_{2+\alpha,2+\alpha}\right)+\mathrm{Lip}\left(\dm{G}\right)\le C_G\spazio,
			$$
			with
			$$
			\mathrm{Lip}\left(\dm{G}\right):=\sup\limits_{m_1\neq m_2}\left(\dw(m_1,m_2)^{-1}\norm{\dm{G}(\cdot,m_1,\cdot)-\dm{G}(\cdot,m_2,\cdot)}_{2+\alpha,2+\alpha}\right)\spazio;
			$$
			\item [(vi)] The following Neumann boundary conditions are satisfied:
			\begin{align*}
				&\left\langle a(y)D_y\dm{F}(x,m,y), \nu(y)\right\rangle_{|\partial\Omega}=0\,,\qquad \left\langle a(y)D_y\dm{G}(x,m,y),\nu(y)\right\rangle_{|\partial\Omega}=0\,,\\
				&\langle a(x)D_xG(x,m), \nu(x)\rangle_{|\partial\Omega}=0\,,
			\end{align*}
			for all $m\in\mathcal{P}(\Omega)$.
		\end{itemize}
	\end{hp}
	
	Some comments about the previous hypotheses: the first five are standard hypotheses in order to obtain existence and uniqueness of solutions for the Mean Field Games system. The hypotheses about the derivative of $F$ and $G$ with respect to the measure will be essential in order to obtain some estimates on a linearized MFG system.
	
	As regards hypotheses $(vi)$, the second and the third boundary conditions are natural compatibility conditions, essential to obtain a classical solution for the $MFG$ and the linearized $MFG$ system. The first boundary condition will be essential in order to prove the Neumann boundary condition of $D_mU$, see Corollary \ref{delarue}.
	
	With these hypotheses we are able to prove existence and uniqueness of a classical solution for the Master Equation \eqref{Master}. The main result of this paper is the following.
	
	\begin{thm}\label{settepuntouno}
		Suppose hypotheses \ref{ipotesi} are satisfied. Then there exists a unique classical solution $U$ of the Master Equation \eqref{Master}.
	\end{thm}
	
	But first, we have to prove some preliminary estimates about the Mean Field Games system and some other estimates on a linearyzed Mean Field Games system, which will be essential in order to ensure the $\mathcal{C}^1$ character of $U$ with respect to $m$.
	
	\section{Preliminary estimates and Mean Field Games system}

	In this section we start giving some technical results for linear parabolic equations, which will be useful in the rest of the Chapter.
	
	Then we will obtain some preliminary estimates for the Master Equation, obtained by a deep analysis of the Mean Field Games related system.
	
	We start with this technical Lemma.
	
	\begin{lem}\label{sonobravo}
		Suppose $a$ satisfies $(i)$ of Hypotheses \ref{ipotesi}, $b,f\in L^\infty(Q_T)$. Furthermore, let $\psi\in\mathcal{C}^{1+\alpha,N}(\Omega)$, with $0\le\alpha<1$. Then the unique solution $z$ of the problem 
		\begin{equation*}
			\begin{cases}
				-z_t-\mathrm{tr}(a(x)D^2z)+b(t,x)\cdot Dz=f(t,x)\,,\\
				z(T)=\psi\,,\\
				aDz\cdot\nu_{|\partial\Omega}=0
			\end{cases}
		\end{equation*}
		satisfies
		\begin{equation}\label{estensione}
			\norm{z}\amu\le C\left(\norminf{f}+\norm{\psi}_{1+\alpha}\right)\,.
		\end{equation}
		\begin{proof}
			Note that, if $f$ and $b$ are continuous bounded functions, with $b$ depending only on $x$, this result is simply \emph{Theorem 5.1.18} of \cite{lunardi}. In the general case, we argue as follows.	
			
			We can write $z=z_1+z_2$, where $z_1$ satisfies
			\begin{equation}\label{problemadue}
				\begin{cases}
					-{(z_1)}_t-\mathrm{tr}(a(x)D^2z_1)=0\,,\\
					z_1(T)=\psi\,,\\
					aDz_1\cdot\nu_{|\partial\Omega}=0\,.
				\end{cases}
			\end{equation}
			and $z_2$ satisfies
			\begin{equation}\label{problemauno}
				\begin{cases}
					-{(z_2)}_t-\mathrm{tr}(a(x)D^2z_2)+b(t,x)\cdot Dz_2=f(t,x)-b(t,x)\cdot Dz_1\,,\\
					z_2(T)=0\,,\\
					aDz_2\cdot\nu_{|\partial\Omega}=0\,,
				\end{cases}
			\end{equation}
			
			Since in the equation \eqref{problemadue} of $z_1$ we do not have a drift term depending on time, we can apply \emph{Theorem 5.1.18} of \cite{lunardi} and obtain
			$$
			\norm{z_1}\amu\le C\norm{\psi}_{1+\alpha}\,.
			$$
			
			As regards \eqref{problemauno}, obviously $z_2(T)\in W^{2,p}(\Omega)$ $\forall p$, and from the estimate of $z_1$ we know that $f-bDz_1\in L^\infty$. So we can apply the Corollary of \emph{Theorem IV.9.1} of \cite{lsu} to obtain that, $\forall r\ge\frac{d+2}{2}\,,$
			\begin{equation*}
				\norm{z_2}_{1-\frac{d+2}{2r},2-\frac{d+2}{r}}\le C\norminf{f-bDz_1}\le C\left(\norminf{f}+\norm{\psi}_{1+\alpha}\right)\,.
			\end{equation*}
			Choosing $r=\frac{d+2}{1-\alpha}$, one has $2-\frac{d+2}{r}=1+\alpha$, and so \eqref{estensione} is satisfied for $z_2$.
			
			Since $z=z_1+z_2$, estimate \eqref{estensione} holds also for $z$. This concludes the proof.
		\end{proof}
	\end{lem}
	
	If the data $f=0$, we can generalize the result of Lemma \ref{sonobravo} if $\psi$ is only a Lipschitz function.
	
	This result is well-known if $a\in\mathcal{C}^2(\Omega)$, by applying a classical Bernstein method. In our framework, we have the following result.
	
	\begin{lem}\label{davverotecnico}
		Suppose $a$ and $b$ be bounded continuous functions, and $\psi\in W^{1,\infty}(\Omega)$. Then the unique solution $z$ of the problem
		\begin{equation}\label{bohmovrim}
			\begin{cases}
				-z_t-\mathrm{tr}(a(x)D^2z)+b(t,x)\cdot Dz=0\,,\\
				z(T)=\psi\,,\\
				aDz\cdot\nu_{|\partial\Omega}=0
			\end{cases}
		\end{equation}
		satisfies a H\"{o}lder condition in $t$ and a Lipschitz condition in $x$, namely $\exists C$ such that
		\begin{align}\label{nnavonmmna}
			|z(t,x)-z(s,x)|\le C\norm{\psi}_{W^{1,\infty}}|t-s|^\miezz\,,\qquad|z(t,x)-z(t,y)|\le C\norm{\psi}_{W^{1,\infty}}|x-y|\,.
		\end{align}
		\begin{proof}
			If $\psi\in\mathcal{C}^{1,N}$, estimates \eqref{nnavonmmna} is guaranteed by \eqref{estensione} of Lemma \ref{sonobravo}.
			
			In the general case, we take $\psi^n\in\mathcal{C}^{1}$ such that $\psi^n\to\psi$ in $\mathcal{C}([0,T]\times\Omega)$ and\\$\norm{\psi^n}_{1}\le C\norm{\psi}_{W^{1,\infty}}$, and we want to make a suitable approximation of it in order to obtain a function $\tilde{\psi}^n\in\mathcal{C}^{1,N}$, also converging to $\psi$.\\
			In order to do that, we first need some useful tools.\\
			For $\delta>0$, $d(\cdot)$ the distance function from $\partial\Omega$, $\Omega_\delta=\{ x\in\Omega\,|\,d(x)\ge\delta \}$ and $x\in\Omega\setminus\Omega_{\delta}$, we consider the following ODE in $\R^d$:
			\begin{align}\label{ODE}
				\begin{cases}
					\xi'(t;x)=-a(\xi(t;x))\nu(\xi(t;x))\,,\\
					\xi(0;x)=x\,,
				\end{cases}
			\end{align}
			where $\nu$ is an extension of the outward unit normal in $\Omega\setminus\Omega_{\delta}$. Actually, we know from \cite{cingul} that $$Dd(x)_{|\partial\Omega}=-\nu(x)\,,$$
			so a suitable extension can be $\nu(x)=-Dd(x)\,$.

			Then we consider the corresponding hitting time of $\partial\Omega_\delta$:
			$$
			T(x):=\inf\left\{t\ge0 \,|\, \xi(t;x)\notin\Omega\setminus\Omega_\delta\right\}\,.
			$$
			We have that $T(x)<+\infty$ for each $x\in\Omega\setminus\Omega_{\delta}$. To prove that, we consider the auxiliary function
			$$
			\Phi(t,x)=\delta-d(\xi(t;x)).
			$$
			So, the function $T(x)$ can be rewritten as
			$$
			T(x)=\inf\left\{ t\ge0\,|\, \Phi(t,x)=0 \right\}\,,
			$$
			and his finiteness is an obvious consequence of the decreasing character of $\Phi$ in time:
			$$
			\partial_t\Phi(t,x)=-Dd(\xi(t;x))\cdot\xi'(t;x)=-\langle a(\xi(t;x))\nu(\xi(t;x)),\nu(\xi(t;x))\rangle\le-\lambda<0\,.
			$$
			Moreover, thanks to Dini's theorem we obtain that $T(x)$ is a $\mathcal{C}^1$ function and his gradient is given by
			$$
			\nabla T(x)=-\frac{\nabla_x\Phi(T(x),x)}{\partial_t\Phi(T(x),x)}=\frac{\nu(\xi(T(x);x))\mathrm{Jac}_x\xi(T(x);x)}{\langle a\nu,\nu\rangle (\xi(T(x);x))}\,.
			$$
			Actually, thanks to the regularity of $a$ and $\Omega$, we can differentiate w.r.t. $x$ the ODE \eqref{ODE} and obtain that $\xi(t;\cdot)\in\mathcal{C}^{1}$.
			
			Now we define the approximating functions $\tilde{\psi}^n$ in the following way:
			\begin{equation}\label{psitildan}
				\tilde{\psi}^n(x)=
				\left\{
				\begin{array}{lll}
					\psi^n(x)\quad & \mbox{if } & x\in\Omega_\delta\,,\\
					\psi^n(\xi(T(x);x))\quad & \mbox{if } & x\in\Omega\setminus\Omega_\delta\,,
				\end{array}
				\right.
			\end{equation}
			eventually considering a $\mathcal{C}^{1}$ regularization in $\Omega_{\delta}\setminus{\Omega_{2\delta}}\,$.
			
			From the definition of $\tilde{\psi}^n$ and the $\mathcal{C}^1$ regularity of $\xi$ and $T$ we have $\tilde{\psi}^n\in\mathcal{C}^1$ and
			$$
			\|{\tilde{\psi}^n}\|_{1}\le C\norm{\psi^n}_1\le C\norm{\psi}_{W^{1,\infty}}\,.
			$$
			Moreover, since near the boundary $\tilde{\psi}^n$ is constant along the trajectories $a(\cdot)\nu(\cdot)$, we have that on $\partial\Omega$
			$$
			a(x)D{\tilde{\psi^n}}(x)\cdot\nu(x)_{|\partial\Omega}=\frac{\partial\tilde{\psi^n}}{\partial (a\nu(x))}(x)_{|\partial\Omega}=0\,,
			$$
			so $\tilde{\psi}^n\in\mathcal{C}^{1,N}$.
			
			Now we consider $z^n$ as the solution of \eqref{bohmovrim} with $\psi$ replaced by $\tilde{\psi}^n$. Then Lemma \ref{sonobravo} implies that $\tilde{\psi}^n$ satisfies
			$$
			\norm{z^n}_{\miezz,1}\le C\|{\tilde{\psi}}\|_1\le C\norm{\psi}_{W^{1,\infty}}\,.
			$$
			Then, Ascoli-Arzel\`a's Theorem tells us that $\exists z$ such that $z^n\to z$ in $\mathcal{C}([0,T]\times\Omega)$. Passing to the limit in the weak formulation of $z^n$, we obtain that $z$ is the unique solution of \eqref{bohmovrim}.
			
			Finally, since $z^n$ satisfies \eqref{nnavonmmna}, we can pass to the pointwise limit when $n\to+\infty$ and obtain the estimate \eqref{nnavonmmna} for $z$. This concludes the Lemma.
		\end{proof}
	\end{lem}
	
	Now we start with the first estimates for the Master Equation.
	
	The first result is obtained by the study of some regularity properties of the $MFG$ system, uniformly in $m_0$.
	
	\begin{prop}
		The system \eqref{meanfieldgames} with conditions \eqref{fame} has a unique classical solution $(u,m)\in \mathcal{C}^{1+\frac{\alpha}{2},2+\alpha}\times \mathcal{C}([0,T];\mathcal{P}(\Omega))$, and this solution satisfies
		\begin{equation}\label{first}
			\sup\limits_{t_1\neq t_2}\frac{\daw{1}(m(t_1),m(t_2))}{|t_1-t_2|^\frac 1 2}+\norm{u}_{1+\frac{\alpha}{2},2+\alpha}\le C\spazio,
		\end{equation}
		where $C$ does not depend on $(t_0,m_0)$.\\
		Furthermore, $m(t)$ has a positive density for each $t>0$ and, if $m_0\in\mathcal{C}^{2+\alpha}$ and satisfies the Neumann boundary condition
		\begin{equation}\label{neumannmzero}
			\left(a(x)Dm_0+(\tilde{b}(0,x)+H_p(x,Du(0,x)))m_0\right)\cdot\nu_{|\partial\Omega}=0\,,
		\end{equation}
		then $m\in\mathcal{C}^{1+\frac{\alpha}{2},2+\alpha}$.\\
		Finally, the solution is stable: if $m_{0n}\to m_0$ in $\mathcal{P}(\Omega)$, then there is convergence of the corresponding solutions of \eqref{meanfieldgames}-\eqref{fame}: $(u_n,m_n)\to (u,m)$ in $\mathcal{C}^{1,2}\times\mathcal{C}([0,T];\mathcal{P}(\Omega))$.
		\begin{proof}
			We use a Schauder fixed point argument.\\
			Let $X\subset\mathcal{C}([t_0,T];\mathcal{P}(\Omega))$ be the set
			$$
			X:=\left\{m\in\mathcal{C}([t_0,T];\mathcal{P}(\Omega))\mbox{ s.t. }\daw{1}(m(t),m(s))\le L|t-s|^\miezz\ \forall s,t\in[t_0,T] \right\}\spazio,
			$$
			where $L$ is a constant that will be chosen later.\\
			It is easy to prove that $X$ is a convex compact set for the uniform distance.\\
			We define a map $\Phi:X\to X$ as follows.\\
			Given $\beta\in X$, we consider the solution of the following Hamilton-Jacobi equation
			\begin{equation}\label{hj}
				\begin{cases}
					-u_t-\mathrm{tr}(a(x)D^2u)+H(x,Du)=F(x,\beta(t))\,,\\
					u(T)=G(x,\beta(T))\,,\\
					a(x)Du\cdot\nu(x)_{|\partial\Omega}=0\,.
				\end{cases}
			\end{equation}
			Thanks to hypothesis $(iv)$ of \ref{ipotesi}, we have ${F(\cdot,\beta(\cdot))}\in\mathcal{C}^{\frac{\alpha}2,\alpha}$ and its norm is bounded by a constant independent of $\beta$. For the same reason $G(\cdot,\beta(T))\in\mathcal{C}^{2+\alpha}$.\\
			It is well known that these hypotheses guarantee the existence and uniqueness of a classical solution. A proof can be found in \cite{lsu}, \textit{Theorem V.7.4}.\\
			So, we can expand with Taylor formula the gradient term and obtain a linear equation satisfied by $u$:
			\begin{align*}
				\begin{cases}
					-u_t-\mathrm{tr}(a(x)D^2u)+H(x,0)+V(t,x)\cdot Du=F(x,\beta(t))\,,\\
					u(T)=G(x,\alpha(T))\,,\\
					a(x)Du\cdot\nu_{\partial\Omega}=0\,.
				\end{cases}
			\end{align*}
			with
			$$
			V(t,x):=\int_0^1 H_p(x,\lambda Du(t,x))\spazio d\lambda\spazio.
			$$
			Thanks to the Lipschitz hypothesis on $H$, $(ii)$ of \ref{ipotesi}, we know that $V\in L^\infty$. So, we can use the Corollary of \emph{Theorem IV.9.1} of \cite{lsu} to obtain
			$$
			Du\in\mathcal{C}^{\frac{\alpha}{2},\alpha}\implies V\in\mathcal{C}^{\frac{\alpha}2,\alpha}\spazio.
			$$
			So, we can apply \emph{Theorem IV.5.3} of \cite{lsu} and get
			\begin{align*}
				\norm{u}_{1+\frac{\alpha}{2},2+\alpha}\le C\left(\norm{F}_{\frac{\alpha}{2},\alpha}+\norm{G}_{2+\alpha}\right)\spazio,
			\end{align*}
			where the constant $C$ does not depend on $\beta$, $t_0$, $m_0$.\\
			Now, we define $\Phi(\beta)=m$, where $m\in\mathcal{C}([t_0,T];\mathcal{P}(\Omega))$ is the solution of the Fokker-Planck equation
			\begin{equation}\label{fpk}
				\begin{cases}
					m_t-\mathrm{div}(a(x)Dm)-\mathrm{div}(m(\tilde{b}(x)+H_p(x,Du)))=0\,,\\
					m(t_0)=m_0\,,\\
					\left(a(x)Dm+(\tilde{b}+H_p(x,Du))m\right)\cdot\nu_{|\partial\Omega}=0\,.
				\end{cases}
			\end{equation}
			It is easy to prove that the above equation has a unique solution in the sense of distribution. A proof in a more general case will be given in the next section, in Proposition \ref{peggiodellagerma}. We want to check that $m\in X$.\\
			Thanks to the distributional formulation, we have
			\begin{equation}\begin{split}\label{sotis}
					&\into \phi(t,x)m(t,dx)-\into\phi(s,x)m(s,dx)\\\,+&\int_s^t\into(-\phi_t-\mathrm{tr}(a(x)D^2\phi)+H_p(x,Du)\cdot D\phi)m(r,dx)dr=0\spazio,
			\end{split}\end{equation}
			for each $\phi\in L^{\infty}$ satisfying in the weak sense
			$$
			\begin{cases}
				-\phi_t-\mathrm{tr}(a(x)D^2\phi)+H_p(x,Du)\cdot D\phi\in L^\infty(Q_T)\\
				aD\phi\cdot\nu_{|\partial\Omega}=0
			\end{cases}\spazio.
			$$
			Take $\psi(\cdot)$ a $1$-Lipschitz function in $\Omega$. So, we choose $\phi$ in the weak formulation as the solution in $[t,T]$ of the following linear equation
			\begin{equation}\label{coglia}\begin{cases}
					-\phi_t-\mathrm{tr}(a(x)D^2\phi)+H_p(x,Du)\cdot D\phi=0\,,\\
					\phi(t)=\psi\,,\\
					a(x)D\phi\cdot\nu_{|\partial\Omega}=0\,.
				\end{cases}
			\end{equation}
			Thanks to Lemma \ref{davverotecnico}, we know that $\phi(\cdot,x)\in\mathcal{C}^{\miezz}([0,T])$ and its H\"{o}lder norm in time is bounded uniformly if $\psi$ is $1$-Lipschitz.\\
			Coming back to \eqref{sotis}, we obtain
			\begin{align*}
				\into\psi(x)(m(t,dx)-m(s,dx))=\into(\phi(t,x)-\phi(s,x))m(s,dx)\le C|t-s|^\miezz\spazio,
			\end{align*}
			and taking the $\sup$ over the $\psi$ $1$-Lipschitz,
			$$
			\dw(m(t),m(s))\le C|t-s|^\miezz\spazio.
			$$
			Choosing $L=C$, we have proved that $m\in X$.\\
			Since $X$ is convex and compact, to apply Schauder's theorem we only need to show the continuity of $\Phi$.\\
			Let $\beta_n\to\beta$, and let $u_n$ and $m_n$ the solutions of \eqref{hj} and \eqref{fpk} related to $\beta_n$. Since $\{u_n\}_n$ is uniformly bounded in $\mathcal{C}^{1+\frac{\alpha}{2},2+\alpha}$, from Ascoli-Arzel\`a's Theorem we have $u_n\to u$ in $\mathcal{C}^{1,2}$.\\
			To prove the convergence of $\{m_n\}_n$, we take $\phi_n$ as the solution of \eqref{coglia} with $Du$ replaced by $Du_n$. Then, as before, $\{\phi_n\}_n$ is a Cauchy sequence in $\mathcal{C}^1$. Actually, the difference $\phi_{n,m}:=\phi_n-\phi_m$ satisfies
			$$
			\begin{cases}
				-(\phi_{n,m})_t-\mathrm{tr}(a(x)D^2\phi_{n,m})+H_p(x,Du_n)\cdot D\phi_{n,m}=(H_p(x,Du_m)-H_p(x,Du_n))\cdot D\phi_m\,,\\
				\phi_{n,m}(t)=0\,,\\
				\bdone{\phi_{n,m}}\,,
			\end{cases}
			$$
			and so Lemma \ref{sonobravo} implies
			$$
			\norm{\phi_{n,m}}\amu\le C\norminf{(H_p(x,Du_m)-H_p(x,Du_n))\cdot D\phi_m}\le C\norminf{Du_m-Du_n}\le \omega(n,k)\,,
			$$
			where $\omega(n,k)\to 0$ when $n,k\to\infty$, and where we use Lemma \ref{davverotecnico} in order to bound $D\phi_m$ in $L^\infty$, without compatibility conditions.\\
			Using \eqref{sotis} with $(m_n,\phi_n)$ and $(m_k,\phi_k)$, for $n,k\in\mathbb{N}$, $s=0$, and subtracting the two equalities, we get
			\begin{align*}
				\into\psi(x)(m_n(t,dx)-m_k(t,dx))=\into(\phi_n(0,x)-\phi_k(0,x))m_0(dx)\le\omega(n,k)\,.
			\end{align*}
			Taking the sup over the $\psi$ $1$-Lipschitz and over $t\in[0,T]$, we obtain
			\begin{align*}
				\sup\limits_{t\in[0,T]}\dw(m_n(t)),m_k(t))\le\omega(n,k)\,,
			\end{align*}
			which proves that $\{m_n\}_n$ is a Cauchy sequence in $X$. Then, $\exists m$ such that $m_n\to m$ in $X$.\\
			Passing to the limit in \eqref{fpk}, we immediately obtain $m=\Phi(\beta)$, which conclude the proof of continuity.\\
			So we can apply Schauder's theorem and obtain a classical solution of the problem \eqref{meanfieldgames}-\eqref{fame}. The estimate \eqref{first} follows from the above estimates for \eqref{hj} and \eqref{fpk}.\\
			To prove the uniqueness, let $(u_1,m_1)$, $(u_2,m_2)$ be two solutions of \eqref{meanfieldgames}-\eqref{fame}.\\
			We use inequality \eqref{dopo}, whose proof will be given in the next lemma, with $m_{01}(t_0)=m_{02}(t_0)=m_0$:
			\begin{align*}
				&\intc{t_0}\mathlarger{(}H(x,Du_2)-H(x,Du_1)-H_p(x,Du_1)(Du_2-Du_1)\mathlarger{)}m_1(t,dx)dt\spazio+\\
				+&\intc{t_0}\mathlarger{(}H(x,Du_1)-H(x,Du_2)-H_p(x,Du_2)(Du_1-Du_2)\mathlarger{)}m_2(t,dx)dt\le 0
			\end{align*}
			Since $H$ is strictly convex, the above inequality gives us $Du_1=Du_2$ in the set\\ $\{m_1>0\}\cup\{m_2>0\}$. Then $m_1$ and $m_2$ solve the same Fokker-Planck equation, and for uniqueness we have $m_1=m_2$.\\
			So $F(x,m_1(t))=F(x,m_2(t))$, $G(x,m_1(T))=G(x,m_2(T))$ and $u_1$ and $u_2$ solve the same Hamilton-Jacobi equation, which implies $u_1=u_2$. The proof of uniqueness is complete.\\
			Finally, if $m_0\in\mathcal{C}^{2+\alpha}$ satisfies \eqref{neumannmzero}, then, splitting the divergence terms in \eqref{fpk}, we have
			\begin{equation*}
				\begin{cases}
					m_t-\mathrm{tr}(a(x)D^2m)-m\spazio\mathrm{div}\left(\tilde{b}(x)+H_p(x,Du)\right)-\left(2\tilde{b}(x)+H_p(x,Du)\right)Dm=0\\
					m(t_0)=m_0\\
					\left(a(x)Dm+(\tilde{b}+H_p(x,Du))m\right)\cdot\nu_{|\partial\Omega}=0
				\end{cases}\spazio.
			\end{equation*}
			Then, thanks to \textit{Theorem IV.5.3} of \cite{lsu}, $m$ is of class $\mathcal{C}^{1+\frac{\alpha}{2},2+\alpha}$.\\
			The stability of solutions is obtained in the same way we used for the continuity of $\Phi$. This concludes the proof.
		\end{proof}
	\end{prop}
	With this proposition, we have obtained that
	\begin{equation}\label{firstmaster}
		\sup\limits_{t\in[0,T]}\sup\limits_{m\in\mathcal{P}(\Omega)}\norm{U(t,\cdot,m)}_{2+\alpha}\le C\spazio,
	\end{equation}
	which gives us an initial regularity result for the function $U$.\\
	To complete the previous proposition, we need the following lemma, based on the so-called \textit{Lasry-Lions monotonicity argument}.
	\begin{lem}
		Let $(u_1,m_1)$ and $(u_2,m_2)$ be two solutions of System \eqref{meanfieldgames}-\eqref{fame}, with $m_1(t_0)=m_{01}$, $m_2(t_0)=m_{02}$. Then
		\begin{equation}\begin{split}\label{dopo}
				&\intc{t_0}\mathlarger{(}H(x,Du_2)-H(x,Du_1)-H_p(x,Du_1)(Du_2-Du_1)\mathlarger{)}m_1(t,dx)dt\\
				+&\intc{t_0}\mathlarger{(}H(x,Du_1)-H(x,Du_2)-H_p(x,Du_2)(Du_1-Du_2)\mathlarger{)}m_2(t,dx)dt\\\le-&\into(u_1(t_0,x)-u_2(t_0,x))(m_{01}(dx)-m_{02}(dx))\spazio.
			\end{split}
		\end{equation}
		\begin{proof}
			See \emph{Lemma 3.1.2} of \cite{card}.
		\end{proof}
	\end{lem}
	\section{Lipschitz continuity of $U$}
	\begin{prop}\label{holder}
		Let $(u_1,m_1)$ and $(u_2,m_2)$ be two solutions of system \eqref{meanfieldgames}-\eqref{fame}, with $m_1(t_0)=m_{01}$, $m_2(t_0)=m_{02}$. Then
		\begin{equation}\begin{split}\label{lipsch}
				\norm{u_1-u_2}\amv&\le C\dw(m_{01},m_{02})\,,\\
				\sup\limits_{t\in[t_0,T]}\dw(m_1(t),m_2(t))& \le C\dw(m_{01},m_{02})\,,
		\end{split}\end{equation}
		where $C$ does not depend on $t_0$, $m_{01}$, $m_{02}$. In particular
		\begin{equation*}
			\sup\limits_{t\in[0,T]}\sup_{m_1\neq m_2}\left[\left(\dw(m_1,m_2)\right)^{-1}\norm{U(t,\cdot,m_1)-U(t,\cdot,m_2)}_{2+\alpha}\right]\le C\,.
		\end{equation*}
		\emph{
			So, the solution of the Master Equation is Lipschitz continuous in the measure variable. This will be essential in order to prove the $\mathcal{C}^1$ character of $U$ with respect to $m$.
		}
		\begin{proof}
			For simplicity, we show the result for $t_0=0$.\\
			\textit{First step: An initial estimate.} Thanks to the hypotheses on $H$ and the Lipschitz bound of $u_1$ and $u_2$, \eqref{dopo} implies
			\begin{align*}
				&\intif|Du_1-Du_2|^2(m_1(t,dx)+m_2(t,dx))dt\le\\\le C&\into (u_1(0,x)-u_2(0,x))(m_{01}(dx)-m_{02}(dx))\le C\norm{u_1-u_2}\amu\dw(m_{01},m_{02}).
			\end{align*}
			\textit{Second step: An estimate on $m_1-m_2$}. We call $m:=m_1-m_2$. We take $\phi$ a sufficiently regular function satisfying $aD\phi\cdot\nu=0$, which will be chosen later. By subtracting the weak formulations \eqref{sotis} of $m_1$ and $m_2$ for $s=0$ and for $\phi$ as test function, we obtain
			\begin{equation}\label{immigrato}
				\begin{split}
					&\into\phi(t,x)m(t,dx)+\inti\left(-\phi_t-\mathrm{tr}(a(x)D^2\phi)+H_p(x,Du_1)D\phi\right)m(s,dx)ds+\\+&\inti(H_p(x,Du_1)-H_p(x,Du_2))D\phi\spazio m_2(s,dx)ds=\into\phi(0,x)(m_{01}(dx)-m_{02}(dx))\,.
				\end{split}
			\end{equation}
			We choose $\phi$ as the solution of \eqref{coglia} related to $u_1$, with terminal condition $\psi\in W^{1,\infty}$. Using the Lipschitz continuity of $H_p$ with respect to $p$, we get
			\begin{align*}
				\into \psi(x) m(t,dx)\le C\inti |Du_1-Du_2| m_2(s,dx)ds+C\dw(m_{01},m_{02})\spazio,
			\end{align*}
			since, for Lemma \ref{davverotecnico}, $\phi$ is Lipschitz continuous with a constant bounded uniformly if $\psi$ is $1$-Lipschitz.\\
			Now we use the Young's inequality and the first step to obtain
			\begin{align*}
				\into \psi(x) m(t,dx)\le\spazio&C\left(\inti |Du_1-Du_2|^2 m_2(s,dx)\right)^\miezz+C\dw(m_{01},m_{02})\le\\\le& \spazio C\left(\norm{u_1-u_2}\amu^\miezz\dw(m_{01},m_{02})^\miezz+\dw(m_{01},m_{02})\right)\spazio,
			\end{align*}
			and finally, taking the sup over the $\psi$ $1$-Lipschitz and the over $t\in[0,T]$,
			\begin{equation}\label{secondstep}
				\sup\limits_{t\in[0,T]}\dw(m_1(t),m_2(t))\le C\left(\norm{u_1-u_2}\amu^\miezz\dw(m_{01},m_{02})^\miezz+\dw(m_{01},m_{02})\right)\spazio.
			\end{equation}
			\textit{Third step: Estimate on $u_1-u_2$ and conclusion.} We call $u:=u_1-u_2$. Then $u$ solves the following equation
			\begin{equation*}
				\begin{cases}
					-u_t-\mathrm{tr}(a(x)D^2 u)+V(t,x)Du=f(t,x)\\
					u(T)=g(x)\\
					a(x)Du\cdot\nu_{|\partial\Omega}=0
				\end{cases}\spazio,
			\end{equation*}
			where
			\begin{align*}
				&V(t,x)=\int_0^1 H_p(x,\lambda Du_1(t,x)+(1-\lambda)Du_2(t,x)d\lambda\spazio;\\
				&f(t,x)=\int_0^1\into\dm{F}(x,\lambda m_1(t)+(1-\lambda)m_2(t),y)(m_1(t,dy)-m_2(t,dy))d\lambda\spazio;\\
				&g(x)=\int_0^1\into\dm{G}(x,\lambda m_1(T)+(1-\lambda)m_2(T),y)(m_1(T,dy)-m_2(T,dy))d\lambda\spazio.
			\end{align*}
			From the regularity of $u_1$ and $u_2$, we have $V(t,\cdot)$ bounded in $\mathcal{C}^{\frac{\alpha}{2},\alpha}$. \\
			We want to apply \emph{Theorem 5.1.21} of \cite{lunardi}. To do this, we have to estimate
			$
			\sup\limits_t\norm{f(t,\cdot)}_\alpha
			$
			
			First, we call
			$$
			m_\lambda(\cdot):=\lambda m_1(\cdot)+(1-\lambda)m_2(\cdot)\spazio.
			$$
			We get
			\begin{align*}
				\sup\limits_{t\in[0,T]}\norm{f(t,\cdot)}_{\alpha}&\le\sup\limits_{t\in[0,T]}\int_0^1\norm{D_y\dm{F}(\cdot,m_\lambda(t),\cdot)}_{\alpha,\infty}d\lambda\,\dw(m_1(t),m_2(t))\\
				&\le C\sup\limits_{t\in[0,T]}\dw(m_1(t),m_2(t))\spazio,
			\end{align*}
			where $C$ depends on the constant $C_F$ in hypotheses \ref{ipotesi}.\\
			In the same way
			\begin{align}
				\norm{g(\cdot)}_{2+\alpha}\le C\sup\limits_{r\in[0,T]}\dw(m_1(r),m_2(r))\,.
			\end{align}
			So we can apply \emph{Theorem 5.1.21} of \cite{lunardi} and obtain
			\begin{equation}\label{finalcountdown}
				\begin{split}
					\norm{u_1-u_2}\amv&\le C\sup\limits_{r\in[0,T]}\dw(m_1(r),m_2(r))\,.
				\end{split}
			\end{equation}
			Coming back to \eqref{secondstep}, this implies
			\begin{align*}
				\sup\limits_{t\in[0,T]}&\dw(m_1(t),m_2(t))\le\\&\le C\left(\left(\sup\limits_{r\in[0,T]}\dw(m_1(r),m_2(r))\right)^{\miezz}\dw(m_{01},m_{02})^\miezz+\dw(m_{01},m_{02})\right)\spazio,
			\end{align*}
			and, using a generalized Young's inequality, this allows us to conclude:
			\begin{align}\label{oterz}
				\sup\limits_{t\in[0,T]}&\dw(m_1(t),m_2(t))\le C\dw(m_{01},m_{02})\,.
			\end{align}
			Plugging this estimate in \eqref{finalcountdown}, we finally obtain
			\begin{align*}
				&\norm{u_1-u_2}\amv\le C\dw(m_{01},m_{02})\label{osicond}\,.\\
			\end{align*}
		\end{proof}
	\end{prop}
	\section{Linearized system and differentiability of $U$ with respect to the measure}
	The proof of existence and uniqueness of solutions for the Master Equation strongly relies on the $\mathcal{C}^1$ character of $U$ with respect to $m$.\\
	The definition of the derivative $\frac{\delta U}{\delta m}$ is strictly related to the solution $(v,\mu)$ of the following \emph{linearized system}:
	\begin{equation}\label{linDuDm}
		\begin{cases}
			-v_t-\mathrm{tr}(a(x)D^2v)+H_p(x,Du)\cdot Dv=\mathlarger{\frac{\delta F}{\delta m}}(x,m(t))(\mu(t))\,,\\
			\mu_t-\mathrm{div}(a(x)D\mu)-\mathrm{div}(\mu(H_p(x,Du)+\tilde{b}))-\mathrm{div}(mH_{pp}(x,Du)Dv)=0\,,\\
			v(T,x)=\mathlarger{\frac{\delta G}{\delta m}}(x,m(T))(\mu(T))\,,\qquad \mu(t_0)=\mu_0\,,\\
			a(x)Dv\cdot\nu_{|\partial\Omega}=0\,,\hspace{1cm}\left(a(x)D\mu+\mu(H_p(x,Du)+\tilde{b})+mH_{pp}(x,Du)Dv\right)\cdot\nu_{|\partial\Omega}=0\,,
		\end{cases}
	\end{equation}
	where we use the notation
	$$
	{\dm{F}}(x,m(t))(\rho(t)):=\left\langle{\dm{F}}(x,m(t),\cdot),\rho(t)\right\rangle
	$$
	and the same for $G$.\\
	
	We want to prove that this system admits a solution and that the following equality holds:
	
	\begin{equation}\label{reprform}
		v(t_0,x)=\left\langle\frac{\delta U}{\delta m}(t_0,x,m_0,\cdot),\mu_0\right\rangle\,. 
	\end{equation}
	
	First, we have to analyze separately the well-posedness of the Fokker-Planck equation in distribution sense:
	\begin{equation}\label{linfp}
		\begin{cases}
			\mu_t-\mathrm{div}(a(x)D\mu)-\mathrm{div}(\mu b)=f\,,\\
			\mu(0)=\mu_0\,,\\
			\left(a(x)D\mu+\mu b\right)\cdot\nu_{|\partial\Omega}=0\,,
		\end{cases}
	\end{equation}
	where $f\in L^1(W^{-1,\infty})$, $\mu_0\in\mathcal{C}^{-(1+\alpha)}$, $b\in L^\infty$.
	
	A suitable distributional definition of solution is the following:
	\begin{defn}\label{canzonenuova}
		Let $f\in L^1(W^{-1,\infty})$, $\mu_0\in\mathcal{C}^{-(1+\alpha)}$, $b\in L^\infty$. We say that a function $\mu\in\mathcal{C}([0,T];\mathcal{C}^{-(1+\alpha),N})\cap L^1(Q_T)$ is a weak solution of \eqref{linfp} if, for all $\psi\in L^\infty(\Omega)$, $\xi\in\mathcal{C}^{1+\alpha,N}$ and $\phi$ solution in $[0,t]\times\Omega$ of the following linear equation
		\begin{equation}\label{hjbfp}
			\begin{cases}
				-\phi_t-\mathrm{div}(aD\phi)+bD\phi=\psi\,,\\
				\phi(t)=\xi\,,\\
				aD\phi\cdot\nu_{|\partial\Omega}=0\,,
			\end{cases}
		\end{equation}
		the following formulation holds:
		\begin{equation}\label{weakmu}
			\langle \mu(t),\xi\rangle+\inti\mu(s,x)\psi(s,x)\,dxds=\langle\mu_0,\phi(0,\cdot)\rangle+\int_0^t\langle f(s),\phi(s,\cdot) \rangle\,ds\,,
		\end{equation}
		where $\langle \cdot,\cdot\rangle$ denotes the duality between $\mathcal{C}^{-(1+\alpha),N}$ and $\mathcal{C}^{1+\alpha,N}$ in the first case, between $\mathcal{C}^{-(1+\alpha)}$ and $\mathcal{C}^{1+\alpha}$ in the second case and between $W^{-1,\infty}$ and $W^{1,\infty}$ in the last case.
	\end{defn}
	
	We note that the definition is well-posed. Actually, $\phi(s,\cdot)$ is in $\mathcal{C}^{1+\alpha}$ $\forall s$ thanks to Lemma \ref{sonobravo}, so $ \langle\mu_0,\phi(0,\cdot)\rangle$ and $\langle f(s),\phi(s,\cdot)\rangle$ are well defined. Moreover, we have
	$$
	\norm{\phi(s,\cdot)}_{W^{1,\infty}}\le C\,.
	$$
	Hence, since $f\in L^1(W^{-1,\infty})$, the last integral is well defined too.
	
	\begin{rem} We are mainly interested in a particular case of distribution $f$. If there exists an integrable function $c:[0,T]\times\Omega\to\R^n$ such that $\forall\phi\in W^{1,\infty}$
		$$
		\langle f(t),\phi\rangle=\into c(t,x)\cdot D\phi(x)\,dx\,,
		$$
		then we can write the problem \eqref{linfp} in this way:
		\begin{equation*}
			\begin{cases}
				\mu_t-\mathrm{div}(a(x)D\mu)-\mathrm{div}(\mu b)=\mathrm{div}(c)\,,\\
				\mu(0)=\mu_0\,,\\
				\left(a(x)D\mu+\mu b+c\right)\cdot\nu_{|\partial\Omega}=0\,,
			\end{cases}
		\end{equation*}
		writing $f$ like a divergence and adjusting the Neumann condition, in order to make sense out of the integration by parts in the regular case. 
		
		In this case, in order to ensure the condition $f\in L^1(W^{-1,\infty})$, we can simply require $c\in L^1(Q_T)$. Actually we have, using Jensen's inequality,
		\begin{align*}
			\norm{f}\amf=\int_0^T\sup\limits_{\norm{\phi}_{W^{1,\infty}}\le 1}\left(\into c(t,x)\cdot D\phi(x)\,dx\right)dt\le C\intif|c(t,x)|\,dxdt=\norm{c}_{L^1}\,,
		\end{align*}
		where $|\cdot|$ is any equivalent norm in $\R^d$.
	\end{rem}
	\vspace{0.42cm}
	
	The next Proposition gives us an exhaustive existence and uniqueness result for \eqref{linfp}.
	\begin{prop}\label{peggiodellagerma}
		Let $f\in L^1(W^{-1,\infty})$, $\mu_0\in\mathcal{C}^{-(1+\alpha)}$, $b\in L^\infty$. Then there exists a unique solution of the Fokker-Planck equation \eqref{linfp}.
		
		This solution satisfies
		\begin{equation}\label{stimefokker}
			\sup_t\norm{\mu(t)}\amc+\norm{\mu}_{L^p}\le C\left(\norm{\mu_0}_{-(1+\alpha)}+\norm{f}_{L^1(W^{-1,\infty})}\right)\,,
		\end{equation}
		where $p=\frac{d+2}{d+1+\alpha}\,$.

		Finally, the solution is stable: if $\mu^n_0\to\mu_0$ in $\mathcal{C}^{-(1+\alpha)}$, $\{b^n\}_n$ uniformly bounded and $b^n\to b$ in $L^p$ $\forall\,p$, $f^n\to f$ in $L^1(W^{-1,\infty})$, then, calling $\mu^n$ and $\mu$ the solutions related, respectively, to $(\mu^n_0,b^n,f^n)$ and $(\mu_0,b,f)$, we have $\mu^n\to\mu$ in $\mathcal{C}([0,T];\mathcal{C}^{-(1+\alpha),N})\cap L^p(Q_T)$.
		\begin{proof}
			
			For the existence part, we start assuming that $f$, $b$, $\mu_0$ are smooth functions, and that $\mu_0$ satisfies
			\begin{equation}\label{neumannmu}
				\left(a(x)D\mu_0+\mu_0 b\right)\cdot\nu_{|\partial\Omega}=0\,.
			\end{equation}
			In this case, we can split the divergence terms in \eqref{linfp} and obtain that $\mu$ is a solution of a linear equation with smooth coefficients. So the existence of solutions in this case is a straightforward consequence of the classical results in \cite{lsu}, \cite{lunardi}.
			
			We consider the unique solution $\phi$ of \eqref{hjbfp} with $\psi=0$ and $\xi\in\mathcal{C}^{1+\alpha,N}$. Multiplying the equation of $\mu$ for $\phi$ and integrating by parts in $[0,t]\times\Omega$ we obtain
			\begin{equation}\label{rhs}
				\langle \mu(t),\xi\rangle=\langle\mu_0,\phi(0,\cdot)\rangle+\int_0^t\langle f(s),\phi(s,\cdot) \rangle\,ds\,.
			\end{equation}
			
			Thanks to Lemma \ref{sonobravo}, we know that
			\begin{equation}\label{upa}
				\norm{\phi}\amu\le C\norm{\xi}_{1+\alpha}\,.
			\end{equation}
			
			Then the right hand side term of \eqref{rhs} is bounded in this way:
			\begin{equation*}
				\langle\mu_0,\phi(0,\cdot)\rangle+\int_0^t\langle f(s),\phi(s,\cdot) \rangle\,ds\le C\norm{\xi}_{1+\alpha}\left(\norm{\mu_0}_{-(1+\alpha)}+
				\int_0^t\norm{f(s)}_{W^{1,\infty}}\right)\,.
			\end{equation*}
			
			Coming back to \eqref{rhs} and passing to the $sup$ when $\xi\in\mathcal{C}^{1+\alpha,N}$, $\norm{\xi}_{1+\alpha}\le 1$, we obtain
			\begin{equation}
				\sup\limits_t\norm{\mu(t)}\amc\le C\left(\norm{\mu_0}_{-(1+\alpha)}+\norm{f}\amf\right)\,.
			\end{equation}
			Now we have to prove the $L^p$ estimate. We consider the solution of \eqref{hjbfp} with $t=T$, $\xi=0$ and $\psi\in L^r$, with $r>d+2$ (we recall that in this chapter we call $d$ the dimension of the space.).
			
			Then the Corollary of \emph{Theorem IV.9.1} of \cite{lsu} tells us that
			\begin{equation}\label{napule}
				\norm{\phi}_{1-\frac{d+2}{2r},2-\frac{d+2}{r}}\le C\norm{\psi}_{L^r}\,.
			\end{equation}
			
			Choosing $r=\frac{d+2}{1-\alpha}$, one has $2-\frac{d+2}{r}=1+\alpha$. Integrating in $[0,T]\times\Omega$ the equation of $\mu$ one has
			$$
			\intif\mu\psi\,dxds=\langle\mu_0,\phi(0,\cdot)\rangle+\int_0^T\langle f(s),\phi(s,\cdot)\rangle\,ds\,.
			$$
			
			Thanks to \eqref{napule} we can estimate the terms on the right-hand side and obtain
			\begin{equation}\label{minecessita}
				\intif\mu\psi\,dxds\le C\norm{\psi}_{L^r}\left( \norm{\mu_0}_{-(1+\alpha)}+\norm{f}\amf \right)\,.
			\end{equation}
			
			Passing to the $sup$ for $\norm{\psi}_{L^r}\le 1$, we finally get
			$$
			\norm{\mu}_{L^p}\le C\left( \norm{\mu_0}_{-(1+\alpha)}+\norm{f}\amf \right)\,,
			$$
			
			with $p$ defined as the conjugate exponent of $r$, i.e. $p=\frac{d+2}{d+1+\alpha}$.
			
			This proves estimates \eqref{stimefokker} in the regular case.\\
			
			In the general case, we consider suitable smooth approximations $\mu_{0}^k$, $f^k$, $b^k$ converging to $\mu_0$, $f$, $b$ respectively in $\mathcal{C}^{-(1+\alpha),N}$, $L^1(W^{-1,\infty})$ and $L^q(Q_T)$ $\forall q\ge 1$, with $b_k$ bounded uniformly in $k$ and with $\mu_0^k$ satisying \eqref{neumannmu}.
			
			We call $\mu^k$ the related solution of \eqref{linfp}. The above convergences tells us that, for a certain $C$,
			\begin{align*}
				\|{\mu^k_0}\|\amb\le C\norm{\mu_0}\amb\,,\qquad&\norminf{b_k}\le C\norminf{b}\\
				&\|{f^k}\|\amf\le C\norm{f}\amf\,.
			\end{align*}
			
			Then we apply \eqref{stimefokker}, to obtain, uniformly in $k$,
			\begin{equation}\label{blaffoff}
				\sup_t\|{\mu^k(t)}\|\amc+\|{\mu^k}\|_{L^p}\le C\left(\norm{\mu_0}_{-(1+\alpha)}+\norm{f}_{L^1(W^{-1,\infty})}\right)\,,
			\end{equation}

			where $C$ actually depends on $b^k$, but since $b^k\to b$ it is bounded uniformly in $k$.
			
			Moreover, the function $\mu^{k,h}:=\mu^k-\mu^h$ also satisfies \eqref{linfp} with data $b=b^k$,\\$f=f^k-f^h+\mathrm{div}(\mu^h(b^k-b^h))$, $\mu^0=\mu^k_0-\mu^h_0$. Then estimates \eqref{stimefokker} tell us that
			\begin{align*}
				&\sup_t\|{\mu^{k,h}(t)}\|\amc\,+\,\|{\mu^{k,h}}\|_{L^p}\\\le C&\left(\|{\mu^k_0-\mu^h_0}\|_{-(1+\alpha)}+\|{f^k-f^h}\|_{L^1(W^{-1,\infty})}+\|{\mathrm{div}(\mu^h(b^k-b^h))}\|\amf\right)\,,
			\end{align*}
			
			The first two terms in the right-hand side easily go to $0$ when $h,k\to+\infty$, since $\mu^k_0$ and $f^k$ are Cauchy sequences. As regards the last term, calling $p'$ the conjugate exponent of $p$, we have
			\begin{align}\label{luigicoibluejeans}
				\|{\mathrm{div}(\mu^h(b^k-b^h))}\|\amf\le C\intif\left|\mu^h(b^k-b^h)\right|\,dxdt\le C\|{b^k-b^h}\|_{L^{p'}}\,,
			\end{align}
			since $\mu^k$ is bounded in $L^p$ by \eqref{blaffoff} (here $C$ depends also on $\mu_0$ and $f$). So, also the last term goes to $0$ since $b^k$ is a Cauchy sequence in $L^q$ $\forall q\ge1$.
			
			Hence, $\{\mu^k\}_k$ is a Cauchy sequence, and so there exists $\mu\in \mathcal{C}([0,T];\mathcal{C}^{-(1+\alpha),N})\cap L^p(Q_T)$ such that
			$$
			\mu^k\to\mu\qquad\mbox{strongly in }\mathcal{C}([0,T];\mathcal{C}^{-(1+\alpha),N})\,,\mbox{ strongly in } L^p(Q_T)\,.
			$$
			Furthermore, $\mu$ satisfies \eqref{stimefokker}.
			
			To conclude, we have to prove that $\mu$ is actually a solution of \eqref{linfp} in the sense of Definition \ref{canzonenuova}.
			
			We take $\phi$ and $\phi^k$ as the solutions of \eqref{hjbfp} related to $b$ and $b^k$. The weak formulation for $\mu^k$ implies that
			
			\begin{equation*}
				\langle \mu^k(t),\xi\rangle+\inti\mu^k(s,x)\psi(s,x)\,dxds=\langle\mu_0^k,\phi^k(0,\cdot)\rangle+\int_0^t\langle f^k(s),\phi^k(s,\cdot) \rangle\,ds\,,
			\end{equation*}
			
			We can immediately pass to the limit in the left-hand side, using the convergence of $\mu^k$ previously obtained.
			
			For the right-hand side, we first need to prove the convergence of $\phi^k$ towards $\phi$. This is immediate: actually, the function $\tilde{\phi}^k:=\phi^k-\phi$ satisfies

			\begin{equation*}
				\begin{cases}
					-\tilde{\phi}^k_t-\mathrm{div}(aD\tilde{\phi}^k)+b^kD\tilde{\phi}^k=(b^k-b)D\phi\,,\\
					\tilde{\phi}^k(t)=0\,,\\
					aD\tilde{\phi}^k\cdot\nu_{|\partial\Omega}=0\,.
				\end{cases}
			\end{equation*}
			
			Then, the Corollary of \emph{Theorem IV.9.1} of \cite{lsu} implies, for a certain $q>d+2$ and depending on $\alpha$,
			$$
			\|{\tilde{\phi}^k}\|\amu\le C\|{(b^k-b)D\phi}\|_{L^q}\to0\,,
			$$
			since $D\phi$ is bounded in $L^\infty$ using Lemma \ref{sonobravo}.
			
			Hence, $\phi^k\to\phi$ in $\mathcal{C}^{\frac{1+\alpha} 2,1+\alpha}$. This allows us to pass to the limit in the right-hand side too and prove that \eqref{weakmu} holds true, and so that $\mu$ is a weak solution of \eqref{linfp}. This concludes the existence part.\\
			
			For the uniqueness part, we consider $\mu_1$ and $\mu_2$ two weak solutions of the system. Then, by linearity, the function $\mu:=\mu_1-\mu_2$ is a weak solution of
			$$
			\begin{cases}
				\mu_t-\mathrm{div}(a(x)D\mu)-\mathrm{div}(\mu b)=0\,,\\
				\mu(0)=0\,,\\
				\left(a(x)D\mu+\mu b\right)\cdot\nu_{|\partial\Omega}=0\,.
			\end{cases}
			$$
			Hence, the weak estimation \eqref{weakmu} implies, $\forall \psi\in L^\infty$ and $\forall\xi\in\mathcal{C}^{1+\alpha,N}$,
			$$
			\langle \mu(t),\xi\rangle+\inti\mu(s,x)\psi(s,x)\,dxds=0\,,
			$$
			which implies $$\norm{\mu}_{L^1}=\supo\norm{\mu(t)}_{-(1+\alpha),N}=0$$ and concludes the uniquess part.\\
			
			Finally, the stability part is an easy consequence of the estimates obtained previously. Let $f^n\to f$, $\mu^n_0\to\mu_0$ and $b^n\to b$. Then the function $\tilde{\mu}^n:=\mu^n-\mu$ satisfies \eqref{linfp} with $b\,,\mu_0$ and $f$ replaced by $b^n$, $\mu^n_0-\mu_0$, $f^n-f+\mathrm{div}(\mu(b^n-b))$. Then we use \eqref{stimefokker} to obtain
			\begin{align*}
				&\sup_t\|{\tilde{\mu}^n}\|\amc\,+\,\|{\tilde{\mu}^n}\|_{L^p}\\\le C&\left(\|{\mu^n_0-\mu_0}\|_{-(1+\alpha)}+\|{f^n-f}\|_{L^1(W^{-1,\infty})}+\|{\mathrm{div}(\mu(b^n-b))}\|\amf\right)\,,
			\end{align*}
			The first two terms in the right-hand side go to $0$. For the last term, the same computations of \eqref{luigicoibluejeans} imply
			$$
			\|{\mathrm{div}(\mu(b^n-b))}\|\amf\le C\|{b^n-b}\|_{L^{p'}}\to0\,.
			$$
			Then $\mu^n\to\mu$ in $\mathcal{C}([0,T];\mathcal{C}^{-(1+\alpha),N})\cap L^p(Q_T)$, which concludes the Proposition.
		\end{proof}
	\end{prop}
	
	The last proposition allows us to get another regularity result of $\mu$, when the data $b$ is more regular. This result will be essential in order to improve the regularity of $\dm{U}$ with respect to $y$.
	
	\begin{cor}
		Let $\mu_0\in\mathcal{C}^{-(1+\alpha)}$, $f\in L^1(W^{-1,\infty})$, $b\in\mathcal{C}^{\frac\alpha 2,\alpha}$. Then the unique solution $\mu$ of \eqref{linfp} satisfies
		\begin{equation}\label{forsemisalvo}
			\supo\norm{\mu(t)}_{-(2+\alpha),N}\le C\left(\norm{\mu_0}_{-(2+\alpha)}+\norm{f}_{L^1(W^{-1,\infty})}\right)\,.
		\end{equation}
		\begin{proof}
			We take $\phi$ as the solution of \eqref{hjbfp}, with $\xi\in C^{2+\alpha,N}(\Omega)$ and $\psi=0$. Then we know from the classical results of \cite{lsu}, \cite{lunardi} (it is important here that $b\in\mathcal{C}^{\frac\alpha2,\alpha}$), that
			$$
			\norm{\phi}\amd\le C\norm{\xi}_{2+\alpha}\,.
			$$
			The weak formulation of $\mu$ \eqref{weakmu} tells us that
			$$
			\langle\mu(t),\xi\rangle=\langle\mu_0,\phi(0,\cdot)\rangle+\int_0^T\langle f(s),\phi(s,\cdot)\rangle\,ds\le C\left(\norm{\mu_0}_{-(2+\alpha)}+\norm{f}_{L^1(W^{-1,\infty})}\right)\norm{\xi}_{2+\alpha}\,.
			$$
			Hence, we can pass to the $sup$ for $\xi\in\mathcal{C}^{2+\alpha,N}$ with $\norm{\xi}_{2+\alpha}\le 1$ and obtain \eqref{forsemisalvo}.
		\end{proof}
		
		\begin{rem}
			We stress the fact that \emph{we shall not formulate problem \eqref{linfp} directly with $\mu_0\in\mathcal{C}^{-(2+\alpha)}$.} Actually, the core of the existence theorem is the $L^p$ bound in space-time of $\mu$, and this is obtained by duality, considering test functions $\phi$ with data $\psi\in L^r$. For this function it is not guaranteed that $\phi(0,\cdot)\in\mathcal{C}^{2+\alpha}(\Omega)$, and an estimation like \eqref{minecessita} is no longer possible.
		\end{rem}
		\vspace{0.1cm}
	\end{cor}
	
	We can also obtain some useful estimates for the density function $m$, as stated in the next result.
	
	\begin{cor}\label{samestrategies}
		Let $(u,m)$ be the solution of the MFG system defined in \eqref{meanfieldgames}-\eqref{fame}. Then we have $m\in L^p(Q_T)$ for $p=\frac{d+2}{d+1+\alpha}$, with
		\begin{equation}\label{mlp}
			\norm{m}_{L^p}\le C\norm{m_0}_{-(1+\alpha)}\,.
		\end{equation}
		Furthermore, if $(u_1,m_1)$ and $(u_2,m_2)$ are two solutions of \eqref{meanfieldgames}-\eqref{fame} with initial conditions $m_{01}$ and $m_{02}$, then we have
		\begin{equation}\label{m12p}
			\norm{m_1-m_2}_{L^p(Q_T)}\le C\dw(m_{01},m_{02})\,.
		\end{equation}
		\begin{proof}
			Since $m$ satisfies \eqref{linfp} with $\mu=m_0\in\mathcal{P}(\Omega)\subset\mathcal{C}^{-(1+\alpha)}$, $b=H_p(x,Du)+\tilde{b}\in L^\infty$ and $f=0$, inequality \eqref{mlp} comes from Proposition \ref{peggiodellagerma}.
			
			For the second inequality, we consider $m:=m_1-m_2$. Then $m$ solves the equation
			\begin{equation*}
				\begin{cases}
					m_t-\mathrm{div}(aDm)-\mathrm{div}(m(H_p(x,Du_1)+\tilde{b}))=\mathrm{div}(m_2(H_p(x,Du_2)-H_p(x,Du_1)))\,,\\
					m(t_0)=m_{01}-m_{02}\,,\\
					\left[aDm+m\tilde{b}+m_1H_p(x,Du_1)-m_2H_p(x,Du_2)\right]\cdot\nu_{|\partial\Omega}=0\,,
				\end{cases}
			\end{equation*}
			i.e. $m$ is a solution of \eqref{linfp} with $f=\mathrm{div}(m_2(H_p(x,Du_2)-H_p(x,Du_1)))$, $\mu_0=m_{01}-m_{02}$, $b=H_p(x,Du_1)$. Then estimations \eqref{stimefokker} imply
			$$
			\norm{m_1-m_2}_{L^p(Q_T)}\le C\left(\norm{\mu_0}\amb+\norm{f}\amf\right)\,.
			$$
			We estimate the right-hand side term. As regards $\mu_0$ we have
			$$
			\norm{\mu_0}_{-(1+\alpha)}=\sup\limits_{\norm{\phi}_{1+\alpha}\le 1}\into \phi(x)(m_{01}-m_{02})(dx)\le C\dw(m_{01},m_{02})\,.
			$$
			For the $f$ term we argue in the following way:
			\begin{align*}
				\norm{f}\amf&=\int_0^T\sup\limits_{\norm{\phi}_{W^{1,\infty}}\le 1}\left(\into H_p(x,Du_2)-H_p(x,Du_1)D\phi\,m_2(t,dx)\right)\,dt\\&\le C\norm{u_1-u_2}\amu\le C\dw(m_{01},m_{02})\,,
			\end{align*}
			which allows us to conclude.
		\end{proof}
	\end{cor}
	
	In order to prove the representation formula \eqref{reprform}, we need to obtain some estimates for a more general linearized system of the form
	\begin{equation}\label{linear}
		\begin{cases}
			-z_t-\mathrm{tr}(a(x)D^2z)+H_p(x,Du)Dz=\mathlarger{\dm{F}}(x,m(t))(\rho(t))+h(t,x)\,,\\
			\rho_t-\mathrm{div}(a(x)D\rho)-\mathrm{div}(\rho(H_p(x,Du)+\tilde{b}))-\mathrm{div}(m H_{pp}(x,Du) Dz+c)=0\,,\\
			z(T,x)=\mathlarger{\dm{G}}(x,m(T))(\rho(T))+z_T(x)\,,\qquad\rho(t_0)=\rho_0\,,\\
			a(x)Dz\cdot\nu_{|\partial\Omega}=0\,,\quad\left(a(x)D\rho+\rho(H_p(x,Du)+\tilde{b})+mH_{pp}(x,Du) Dz+c\right)\cdot\nu_{|\partial\Omega}=0\,,
		\end{cases}
	\end{equation}
	where we require
	$$
	z_T\in\mathcal{C}^{2+\alpha},\quad\rho_0\in\mathcal{C}^{-(1+\alpha)},\quad h\in \mathcal{C}^{0,\alpha}([t_0,T]\times\Omega),\quad c\in L^1([t_0,T]\times\Omega)\,.
	$$
	
	Moreover, $z_T$ satisfies
	\begin{equation}\label{neumannzT}
		aDz_T\cdot\nu_{|\partial\Omega}=0\,.
	\end{equation}
	
	A suitable definition of solution for this system is the following:
	\begin{defn}\label{defn}
		We say that a couple $(z,\rho)\in\mathcal{C}^{1,2+\alpha}\times\,\left(\mathcal{C}([0,T];\mathcal{C}^{-(1+\alpha),N}(\Omega))\cap L^1(Q_T)\right)$ is a solution of the equation \eqref{linear} if
		\begin{itemize}
			\item $z$ is a classical solution of the linear equation;
			\item $\rho$ is a distributional solution of the Fokker-Planck equation in the sense of Definition \ref{canzonenuova}.
		\end{itemize}
	\end{defn}
	
	We start with the following existence result.

	\begin{prop}\label{linearD}
		Let hypotheses \ref{ipotesi} hold for $0<\alpha<1$. Then there exists a unique solution $(z,\rho)\in\mathcal{C}^{1,2+\alpha}\times\,\left(\mathcal{C}([0,T];\mathcal{C}^{-(1+\alpha),N}(\Omega))\cap L^1(Q_T)\right)$ of system \eqref{linear}. This solution satisfies, for a certain $p>1$,
		\begin{equation}
			\begin{split}\label{stimelin}
				\norm{z}\amv+\sup\limits_t\norm{\rho(t)}_{-(1+\alpha),N}+\norm{\rho}_{L^p}\le CM\spazio,
			\end{split}
		\end{equation}
		where $C$ depends on $H$ and where $M$ is given by
		\begin{equation}\label{emme}
			M:=\norm{z_T}_{2+\alpha}+\norm{\rho_0}_{-(1+\alpha)}+\norm{h}_{0,\alpha}+\norm{c}_{L^1}\,.
		\end{equation}
		\begin{proof}
			As always, we can assume $t_0=0$ without loss of generality.
			
			The main idea is to apply Schaefer's Theorem.\\
			
			\emph{Step 1: Definition of the map $\mathbf{\Phi}$ satisfying Schaefer's Theorem}.
			We set $X:=\mathcal{C}([0,T];\mathcal{C}^{-(1+\alpha),N})$, endowed with the norm
			\begin{equation*}
				\norm{\phi}_X:=\supo\norm{\phi(t)}_{-(1+\alpha),N}\,.
			\end{equation*}
			
			For $\rho\in X$, we consider the classical solution $z$ of the following equation
			\begin{equation}
				\label{zlin}
				\begin{cases}
					-z_t-\tr{z}+H_p(x,Du)Dz=\mathlarger{\dm{F}}(x,m(t))(\rho(t))+h(t,x)\,,\\
					z(T)=\mathlarger{\dm{G}}(x,m(T))(\rho(T))+z_T\,,\\
					a(x)Dz\cdot\nu_{|\partial\Omega}=0\,.
				\end{cases}
			\end{equation}
			
			We note that, from  Hypotheses \ref{ipotesi}, we have
			$$
			\langle a(x)D_xG(x,m), \nu(x)\rangle_{|\partial\Omega}=0\quad\forall m\in\mathcal{P}(\Omega)\implies\left\langle a(x)D_x\dm{G}(x,m(T))(\mu(T)), \nu(x)\right\rangle_{|\partial\Omega}\!\!\!\!\!\!\!=0\,.
			$$
			Hence, compatibility conditions are satisfied for equation \eqref{zlin} and, from \emph{Theorem 5.1.21} of \cite{lunardi}, $z$ satisfies
			\begin{equation}\label{stimz}
				\begin{split}
					\norm{z}\amv&\le C\left(\norm{z_T}_{2+\alpha}+\supo\norm{\rho(t)}_{-(2+\alpha),N}+\norm{h}_{0,\alpha}\right)\\&\le C\left(M+\supo\norm{\rho(t)}\amc\right)\,,
				\end{split}
			\end{equation}
			where we also use hypothesis $(vi)$ of \ref{ipotesi}, for the boundary condition of $\dm{F}$.
			
			Then we define $\mathbf{\Phi}(\rho):=\tilde{\rho}$, where $\tilde{\rho}$ is the solution in the sense of Definition \ref{canzonenuova} to:
			\begin{equation}
				\label{plin}
				\begin{cases}
					\rt_t-\mathrm{div}(a(x)D\rt)-\mathrm{div}(\rt (H_p(x,Du)+\tilde{b}))-\mathrm{div}(mH_{pp}(x,Du) Dz+c)=0\\
					\rt(0)=\rho_0\\
					\left(a(x)D\rt+\rt(H_p(x,Du)+\tilde{b})+mH_{pp}(x,Du) Dz+c\right)\cdot\nu_{|\partial\Omega}=0
				\end{cases}\spazio.
			\end{equation}
			Thanks to Proposition \ref{peggiodellagerma}, we have $\tilde{\rho}\in X$. We want to prove that the map $\mathbf{\Phi}$ is continuous and compact.\\
			For the compactness, let $\{\rho_n\}_n\subset X$ be a subsequence with $\norm{\rho_n}_X\le{C}$ for a certain $C>0$. We consider for each $n$ the solutions $z_n$ and $\tilde{\rho}_n$ of \eqref{zlin} and \eqref{plin} associated to $\tilde{\rho}_n$.\\
			Using \eqref{stimz}, we have $\norm{z_n}\amv\le C_1$, where $C_1$ depends on $C$. Then, thanks to Ascoli-Arzel\`a's Theorem, and using also \eqref{precisissimongulaeva}, $\exists z$ s.t. $z_n\to z$ up to subsequences at least in $\mathcal{C}([0,T];\mathcal{C}^1(\Omega))$.
			
			Using the pointwise convergence of $Dz_n$ and the $L^p$ boundedness of $m$ stated in \eqref{mlp}, we immediately obtain
			$$
			mH_{pp}(x,Du)Dz_n\,+\,c\to mH_{pp}(x,Du)Dz\,+\,c\qquad\mbox{in }L^1(Q_T)\,,
			$$
			which immediately implies
			$$
			\mathrm{div}(mH_{pp}(x,Du)Dz_n\,+\,c)\to\mathrm{div}{(mH_{pp}(x,Du)Dz\,+\,c)}\qquad\mbox{in }L^1(W^{-1,\infty})\,.
			$$
			Hence, stability results proved in Proposition \ref{peggiodellagerma} proves that $\tilde{\rho}_n\to\tilde{\rho}$ in $X$, where $\tilde{\rho}$ is the solution related to $Dz$. This proves the compactness result.
			
			The continuity of $\Phi$ can be proved used the same computations of the compactness.
			
			Finally, in order to apply Schaefer's theorem, we have to prove that
			$$
			\exists M>0 \mbox{ s.t. } \rho=\sigma\mathbf{\Phi}(\rho)\ \mbox{ and }\sigma\in[0,1]\implies\norm{\rho}_X\le M\spazio.
			$$
			We will prove in the next step that, if $\rho=\sigma\mathbf{\Phi}(\rho)$, then the couple $(z,\rho)$ satisfies \eqref{stimelin}. This allows us to apply Schaefer's theorem and also gives us the desired estimate \eqref{stimelin}, since each solution $(z,\rho)$ of the system satisfies $\rho=\sigma\mathbf{\Phi}(\rho)$ with $\sigma=1$.\\
			
			\emph{Step 2: Estimate of $\rho$ and $z$}. Let $(\rho,\sigma)\in X\times[0,1]$ such that $\rho=\sigma\mathbf{\Phi}(\rho)$. Then the couple $(z,\rho)$ satisfies
			\begin{equation*}
				\begin{cases}
					-z_t-\mathrm{tr}(a(x)D^2z)+H_p(x,Du)Dz=\mathlarger{\dm{F}}(x,m(t))(\rho(t))+h(t,x)\\
					\rho_t-\mathrm{div}(a(x)D\rho)-\mathrm{div}(\rho(H_p(x,Du)+\tilde{b}))-\sigma\mathrm{div}(mH_{pp}(x,Du) Dz+c)=0\\
					z(T,x)=\mathlarger{\dm{G}}(x,m(T))(\rho(T))+z_T(x)\hspace{2cm}\rho(0)=\sigma\rho_0\\
					a(x)Dz\cdot\nu_{|\partial\Omega}=0\hspace{1cm}\left(a(x)D\rho+\rho(H_p(x,Du)+\tilde{b})+\sigma(mH_{pp}(x,Du) Dz+c)\right)\cdot\nu_{|\partial\Omega}=0
				\end{cases}\spazio.
			\end{equation*}
			We want to use $z$ as test function for the equation of $\rho$. This is allowed since $z$ satisfies \eqref{hjbfp} with
			\begin{align*}
				\psi=\dm{F}(x,m(t))(\rho(t))+h(t,x)\in L^\infty(\Omega)\,,\qquad\xi=\dm{G}(x,m(T))(\rho(T))+z_T(x)\in\mathcal{C}^{1+\alpha,N}
			\end{align*}
			
			We obtain from the weak formulation of $\rho$:	
			\begin{equation*}
				\begin{split}
					&\into \left(\rho(T,x)z(T,x)-\sigma\rho_0(x)z(0,x)\right)dx=-\sigma\intif\langle c,Dz\rangle dxdt+\\
					-&\intif\rho(t,x)\left(\dm{F}(x,m(t))(\rho(t))+h\right)dxdt-\sigma\intif m\langle H_{pp}(x,Du) Dz,Dz\rangle\spazio dxdt\,.
				\end{split}
			\end{equation*}
			Using the terminal condition of $z$ and the monotonicity of $F$ and $G$, we get a first estimate:
			\begin{equation}\label{stimasigma}
				\begin{split}
					\sigma\intif m\langle H_{pp}(x,Du) Dz,Dz\rangle\spazio dxdt
					\le&\supo\norm{\rho(t)}_{-(2+\alpha),N}\norm{z_T}_{2+\alpha}+\norm{\rho}_{L^p}\norminf{h}\\
					+&\norm{z}\amv\left(\norm{\rho_0}_{-(2+\alpha),N}+\norm{c}_{L^1}\right)\\\le\,&M \left(\supo\norm{\rho(t)}_{-(1+\alpha),N}+\norm{\rho}_{L^1}+\norm{z}\amv\right)\,.
				\end{split}
			\end{equation}
			We already know an initial estimate on $z$ in \eqref{stimz}. Now we need to estimate $\rho$.
			
			Using \eqref{stimefokker} we obtain
			\begin{equation}\label{duality}
				\supo\norm{\rho}\amc+\norm{\rho}_{L^p}\le C\left(\norm{\sigma mH_{pp}(x,Du)Dz}_{L^1}+\norm{c}_{L^1}+\norm{\rho_0}\amb\right)
			\end{equation}
			
			As regards the first term in the right hand side, we can use H\"{o}lder's inequality and \eqref{stimasigma} to obtain
			\begin{align*}
				&\norm{mH_{pp}(x,Du)Dz}_{L^1}=\sigma\sup\limits_{\substack{\norminf{\phi}\le 1\\\phi\in L^\infty(Q_T;\R^d)}}\intif m\langle H_{pp}(x,Du)Dz,\phi\rangle\,dxdt\\
				&\le\sigma\left(\intif m\langle H_{pp}(x,Du) Dz,Dz\rangle\spazio dxdt\right)^\miezz\left(\intif m\langle H_{pp}(x,Du)\phi,\phi\rangle\spazio dxdt\right)^\miezz\\
				&\le
				M^\miezz\left(\supo\norm{\rho(t)}^\miezz_{-(1+\alpha),N}+\norm{\rho}^\miezz_{L^1}+\norm{z}^\miezz\amv\right)\,
			\end{align*}
			Putting these estimates into \eqref{duality} we obtain
			\begin{align*}
				\supo\norm{\rho}\amc+\norm{\rho}_{L^p}\le C\left( M+M^\miezz\left(\supo\norm{\rho(t)}^\miezz_{-(1+\alpha),N}+\norm{\rho}^\miezz_{L^1}+\norm{z}^\miezz\amv\right)\right)\,.
			\end{align*}
			Using a generalized Young's inequality with suitable coefficients, we get
			\begin{align}\label{stimarho}
				\supo\norm{\rho}\amc+\norm{\rho}_{L^p}\le C\left(M+M^\miezz\norm{z}\amv^\miezz\right)\spazio.
			\end{align}
			This gives us an initial estimate for $\rho$, depending on the estimate of $z$.
			
			Coming back to \eqref{stimz}, \eqref{stimarho} implies
			\begin{align*}
				\norm{z}\amv\le C\left(M+M^\miezz\norm{z}\amv^\miezz\right)\spazio.
			\end{align*}
			Using a generalized Young's inequality with suitable coefficients, this implies
			$$
			\norm{z}\amv\le Cm\,.
			$$
			Plugging this estimate in \eqref{stimarho}, we finally obtain
			\begin{align*}
				\norm{z}\amv+\supo\norm{\rho}\amc+\norm{\rho}_{L^p}\le CM\spazio.
			\end{align*}
			This concludes the existence result.\\\\
			
			\emph{Step 3. Uniqueness}. Let $(z_1,\rho_1)$ and $(z_2,\rho_2)$ be two solutions of \eqref{linear}. Then the couple $(z,\rho):=(z_1-z_2,\rho_1-\rho_2)$ satisfies the following linear system:
			\begin{equation*}
				\begin{cases}
					-z_t-\mathrm{tr}(a(x)D^2z)+H_p(x,Du)Dz=\mathlarger{\dm{F}}(x,m(t))(\rho(t))=0\,,\\
					\rho_t-\mathrm{div}(a(x)D\rho)-\mathrm{div}(\rho(H_p(x,Du)+\tilde{b}))-\mathrm{div}(m H_{pp}(x,Du) Dz)=0\,,\\
					z(T,x)=\mathlarger{\dm{G}}(x,m(T))(\rho(T))\,,\qquad\rho(t_0)=0\,,\\
					a(x)Dz\cdot\nu_{|\partial\Omega}=0\,,\quad\left(a(x)D\rho+\rho(H_p(x,Du)+\tilde{b})+mH_{pp}(x,Du) Dz\right)\cdot\nu_{|\partial\Omega}=0\,,
				\end{cases}
			\end{equation*}
			i.e., a system of the form \eqref{linear} with $h=c=z_T=\rho_0=0$. Then estimation \eqref{stimelin} tells us that
			$$
			\norm{z}\amv+\supo\norm{\rho}\amc+\norm{\rho}_{L^p}\le 0\spazio.
			$$
			and so $z=0$, $\rho=0$. This concludes the Proposition.
		\end{proof}
	\end{prop}

	We are ready to prove that \eqref{linDuDm} has a fundamental solution. This solution will be the desired derivative $\dm{U}$.
	
	\begin{prop}
		Equation \eqref{linDuDm} has a fundamental solution, i.e. there exists a function $K:[0,T]\times\Omega\times\mathcal{P}(\Omega)\times\Omega\to\R$ such that, for any $(t_0,m_0,\mu_0)$ we have
		\begin{equation}\label{repres}
			v(t_0,x)=_{-(1+\alpha)}\!\!\langle \mu_0,K(t_0,x,m_0,\cdot)\rangle_{1+\alpha}
		\end{equation}
		Moreover, $K(t_0,\cdot,m_0,\cdot)\in\mathcal{C}^{2+\alpha}(\Omega)\times \mathcal{C}^{1+\alpha}(\Omega)$ with
		\begin{equation}\label{kappa}
			\sup\limits_{(t,m)\in[0,T]\times\mathcal{P}(\Omega)}\norm{K(t,\cdot,m,\cdot)}_{2+\alpha,1+\alpha}\le C\,,
		\end{equation}
		and the second derivatives w.r.t. $x$ and the first derivatives w.r.t. $y$ are continuous in all variables.
		\begin{proof}
			From now on, we indicate with $v(t,x;\mu_0)$ the solution of the first equation of \eqref{linDuDm} related to $\mu_0$.
			We start considering, for  $y\in\Omega$, $\mu_0=\delta_y$, the Dirac function at $y$. We define
			$$
			K(t_0,x,m_0,y)=v(t_0,x;\delta_y)
			$$
			Thanks to \eqref{stimelin}, one immediately knows that $K$ is twice differentiable w.r.t. $x$ and
			$$
			\norm{K(t_0,\cdot,m_0,y)}_{2+\alpha}\le C\norm{\delta_y}_{-(1+\alpha)}=C
			$$
			Moreover, we can use the linearity of the system \eqref{linear} to obtain
			$$
			\frac{K(t_0,x,m_0,y+he_j)-K(t_0,x,m_0,y)}h=v(t_0,x;\Delta_{h,j}\delta_{y})\,,
			$$
			where  $\Delta_{h,j}\delta_{y}=\frac1h(\delta_{y+he_j}-\delta_y)$. Using stability results for \eqref{linDuDm}, proved previously, we can pass to the limit and find that
			$$
			\frac{\partial K}{\partial y_j}(t_0,x,m_0,y)=v(t_0,x;-\partial_{y_j}\delta_y)\,,
			$$
			where the derivative of the Dirac delta function is in the sense of distribution.
			Since $\partial_{y_i}\delta_y$ is bounded in $\mathcal{C}^{-(1+\alpha)}$ for all $i,j$, from \eqref{stimelin} we deduce that the second derivatives of $K$ with respect to $x$ are well defined and bounded.
			
			The representation formula \eqref{repres} is an immediate consequence of the linear character of the equation and of the density of the set generated by the Dirac functions. This concludes the proof.
		\end{proof}
	\end{prop}
	
	Now we are ready to prove the differentiability of the function $U$ with respect to the measure $m$.
	
	In particular, we want to prove that this fundamental solution $K$ is actually the derivative of $U$ with respect to the measure.
	\begin{thm}
		Let $(u_1,m_1)$ and $(u_2,m_2)$ be two solutions of the Mean Field Games system \eqref{meanfieldgames}-\eqref{fame}, associated with the starting initial conditions $(t_0,m_0^1)$ and $(t_0,m_0^2)$.
		Let $(v,\mu)$ be the solution of the linearized system \eqref{linDuDm} related to $(u_2,m_2)$, with initial condition $(t_0,m_0^1-m_0^2)$. Then we have
		\begin{equation}\label{boundmder}
			\norm{u_1-u_2-v}\amv+\supo\norm{m_1(t)-m_2(t)-\mu(t)}\amc\le C\dw(m_0^1,m_0^2)^{2}\,,
		\end{equation}
		
		Consequently, the function $U$ defined in \eqref{U} is differentiable with respect to $m$.
		
		\begin{proof}
			
			We call $(z,\rho)=(u_1-u_2-v,m_1-m_2-\mu)$. Then $(z,\rho)$ satisfies
			\begin{equation*}
				\begin{cases}
					-z_t-\mathrm{tr}(a(x)D^2z)+H_p(x,Du_2)Dz=\mathlarger{\dm{F}}(x,m_2(t))(\rho(t))+h(t,x)\,,\\
					\rho_t-\mathrm{div}(a(x)D\rho)-\mathrm{div}(\rho(H_p(x,Du_2)+\tilde{b}))-\mathrm{div}(m H_{pp}(x,Du_2) Dz+c)=0\,,\\
					z(T,x)=\mathlarger{\dm{G}}(x,m_2(T))(\rho(T))+z_T(x)\,,\qquad\rho(t_0)=0,,\\
					a(x)Dz\cdot\nu_{|\partial\Omega}=0\,,\quad\left(a(x)D\rho+\rho(H_p(x,Du)+\tilde{b})+mH_{pp}(x,Du) Dz+c\right)\cdot\nu_{|\partial\Omega}=0\,,
				\end{cases}
			\end{equation*}
			\begin{align*}
				&h(t,x)=h_1(t,x)+h_2(t,x)\,,\\
				&h_1=-\int_0^1 (H_p(x,sDu_1+(1-s)Du_2)-H_p(x,Du_2))\cdot D(u_1-u_2)\,ds\,,\\
				&h_2=\int_0^1\!\into\left(\dm{F}(x,sm_1(t)+(1-s)m_2(t),y)-\dm{F}(x,m_2(t),y)\right)(m_1(t)-m_2(t))(dy)ds,\\
				&c(t)=c_1(t)+c_2(t)\,,\\
				&c_1(t)=(m_1(t)-m_2(t))H_{pp}(x,Du_2)(Du_1-Du_2)\,,\\
				&c_2(t)=m_1\int_0^1\left(H_{pp}(x,sDu_1+(1-s)Du_2)-H_{pp}(x,Du_2)\right)(Du_1-Du_2)\,ds\,,\\
				&z_T=\int_0^1\into\left(\dm{G}(x,sm_1(T)+(1-s)m_2(T),y)\right.\\
				&\left.\hspace{6cm}-\dm{G}(x,m_2(T),y)\right)(m_1(T)-m_2(T))(dy)ds\,.
			\end{align*}
			So, \eqref{stimelin} implies that
			\begin{equation}\label{rogueuno}
				\norm{u_1-u_2-v}\amv+\supo\norm{m_1(t)-m_2(t)-\mu(t)}\amc\le C\left(\norm{h}_{0,\alpha}+\norm{c}_{L^1}+\norm{z_T}_{2+\alpha}\right)\,.
			\end{equation}
			
			Now we bound the right-hand side term in order to obtain \eqref{boundmder}.
			
			We start with the term $h=h_1+h_2$. We can write
			$$
			h_1=-\int_0^1\int_0^1 s\, \langle H_{pp}(x,rsDu_1+(1-rs)Du_2)\,(Du_1-Du_2)\,,\, (Du_1-Du_2)\rangle\,drds\,.
			$$

			Using the properties of H\"{o}lder norm and \eqref{lipsch}, it is immediate to obtain
			\begin{align*}
				\norm{h_1}_{0,\alpha}\le C\norm{D(u_1-u_2)}_{0,\alpha}^2\le C\dw(m_{0}^1,m_{0}^2)^2\,.
			\end{align*}
			
			As regards the $h_2$ term, we can immediately bound the quantity
			$$
			|h_2(t,x)-h_2(t,y)|
			$$
			by
			\begin{align*}
				|x-y|^\alpha\dw(m_1(t),m_2(t))\int_0^1\norm{D_m F(\cdot,sm_1(t)+(1-s)m_2(t),\cdot)-D_m F(\cdot,m_2(t),\cdot)}_{\alpha,\infty}ds\,.
			\end{align*}
			Using the regularity of $F$ and \eqref{lipsch}, we get
			$$
			\norm{h_2}_{0,\alpha}=\sup\limits_{t\in[0,T]}\norm{h_2(t,\cdot)}_\alpha\le C\dw(m_0^1,m_0^2)^2\,.
			$$
			A similar estimate holds for the function $z_T$. As regards the function $c$, we have
			\begin{align*}
				\norm{c_1}_{L^1}=\intif H_{pp}(x,Du_2)(Du_1-Du_2)(m_1(t,dx)-m_2(t,dx))\,dt\\ \le C\norm{u_1-u_2}\amv\dw(m_1(t),m_2(t))\le C\dw(m_0^1,m_0^2)^{2}\,,
			\end{align*}
			and, using the notation $u_{1+s}:=sDu_1+(1-s)Du_2$,
			\begin{align*}
				\norm{c_2}_{L^1}=&\int_0^1\intif \left(H_{pp}(x,Du_{1+s})-H_{pp}(x,Du_2)\right)(Du_1-Du_2)m_1(t,dx)\,dtds\\
				\le&\,C\norminf{Du_1-Du_2}^2\le C\dw(m_0^1,m_0^2)^2\,.
			\end{align*}
			Substituting these estimates in \eqref{rogueuno}, we obtain \eqref{boundmder} and we conclude the proof.
		\end{proof}
	\end{thm}
	
	Since
	$$
	v(t_0,x)=\into K(t_0,x,m_{02},y)(m_{01}(dy)-m_{02}(dy))\,,
	$$
	equation \eqref{boundmder} implies
	$$
	\norminf{U(t_0,\cdot,m_{01})-U(t_0,\cdot,m_{02})-\into K(t_0,\cdot,m_{02},y)(m_{01}-m_{02})(dy) }\le C\dw(m_{01},m_{02})^{2}.
	$$
	
	As a straightforward consequence, we have that $U$ is differentiable with respect to $m$ and
	$$
	\dm{U}(t,x,m,y)=K(t,x,m,y)\,.
	$$
	Consequently, using \eqref{kappa} we obtain
	\begin{equation}\label{regdu}
		\sup\limits_t\norm{\dm{U}(t,\cdot,m,\cdot)}_{2+\alpha,1+\alpha}\le C\,.
	\end{equation}
	
	But, in order to make sense to equation \eqref{Master}, we need at least that $\dm{U}$ is almost everywhere twice differentiable with respect to $y$.
	
	To do that, we need to improve the estimates \eqref{stimelin} for a couple $(v,\mu)$ solution of \eqref{linDuDm}.
	
	\begin{prop}
		Let $\mu_0\in\mathcal{C}^{-(1+\alpha)}$. Then the unique solution $(v,\mu)$ satisfies
		\begin{equation}\label{sbrigati}
			\norm{v}_{1,2+\alpha}+\supo\norm{\mu(t)}_{-(2+\alpha),N}\le C\norm{\mu_0}_{-(2+\alpha)}\,.
		\end{equation}
		\begin{proof}
			We consider the solution $(v,\mu)$ obtained in Proposition \ref{linearD}. Since $\mu$ satisfies $\mu=\sigma\Phi(\mu)$ with $\sigma=1$, we can use \eqref{stimz} with $z_T=h=0$ and obtain
			\begin{equation}\label{cumnupnat}
				\norm{v}_{1,2+\alpha}\le C\supo\norm{\mu(t)}_{-(2+\alpha),N}\,.
			\end{equation}
			We want to estimate the right-hand side. Using \eqref{forsemisalvo} we have
			\begin{equation}\label{ngroc}
				\supo\norm{\mu(t)}_{-(2+\alpha),N}\le C\left(\norm{\mu_0}_{-(2+\alpha)}+\norm{ mH_{pp}(x,Du)Dv}_{L^1}\right)\,.
			\end{equation}
			
			The last term is estimated, as in Proposition \ref{linearD}, by
			\begin{equation}\label{mannaggia}
				\norm{\sigma mH_{pp}(x,Du)Dv}_{L^1}\le C\left(\intif m\langle H_{pp}(x,Du)Dv,Dv\rangle\,dxdt\right)^\miezz\,.
			\end{equation}
			The right-hand side term can be bounded using \eqref{stimasigma} with $h=z_T=c=0$:
			\begin{equation}\label{crist}
				\begin{split}
					\intif m\langle H_{pp}(x,Du) Dv,Dv\rangle\spazio dxdt
					\le\norm{v}\amv\norm{\mu_0}_{-(2+\alpha)}\,.
				\end{split}
			\end{equation}
			
			Hence, plugging estimates \eqref{mannaggia} and \eqref{crist} into \eqref{ngroc} we obtain
			\begin{equation}\label{probbiatottquanta}
				\supo\norm{\mu(t)}_{-(2+\alpha),N}\le C\left(\norm{\mu_0}_{-(2+\alpha)}+\norm{v}\amv^\miezz\norm{\mu_0}_{-(2+\alpha),N}^\miezz\right)\,.
			\end{equation}
			Coming back to \eqref{cumnupnat} and using a generalized Young's inequality, we get
			$$
			\norm{v}\amv\le C\norm{\mu_0}_{-(2+\alpha)}\,,
			$$
			and finally, substituting the last estimate into \eqref{probbiatottquanta}, we obtain \eqref{sbrigati} and we conclude.
		\end{proof}
	\end{prop}
	
	As an immediate Corollary, we get the desired estimate for $\dm{U}$.
	\begin{cor}
		Suppose hypotheses \ref{ipotesi} satisfied. Then the derivative $\dm{U}$ is twice differentiable with respect to $y$, together with its first and second derivatives with respect to $x$, and the following estimate hold:
		\begin{equation}\label{lentezza}
			\norm{\dm{U}(t,\cdot,m,\cdot)}_{2+\alpha,2+\alpha}\le C\,.
		\end{equation}
		\begin{proof}
			We want to prove that, $\forall\,i,j$, the incremental ratio
			\begin{equation}\label{alotteriarubabbeh}
				R^h_{i,j}(x,y):=\frac{\partial_{y_i}\dm{U}(t_0,x,m_0,y+he_j)-\partial_{y_i}\dm{U}(t_0,x,m_0,y)}{h}
			\end{equation}
			is a Cauchy sequence for $h\to0\,$ together with its first and second derivatives with respect to $x$. Then we have to estimate, for $h,k>0$, the quantity $\left|D^l_xR^h_{i,j}(x,y)-D^l_xR^k_{i,j}(x,y)\right|\,,$ for $|l|\le 2$.
			
			We already know that
			$$
			\partial_{y_i}\dm{U}(t_0,x,m_0,y)=v(t_0,x;-\partial_{y_i}\delta_y)\,.
			$$
			Using the linearity of the system \eqref{linDuDm}, we obtain that 
			$$
			\left|D^l_xR^h_{i,j}(x,y)-D^l_xR^k_{i,j}(x,y)\right|=D^l_xv\left(t_0,x;\Delta_h^j(-\partial_{y_i}\delta_y)-\Delta_k^j(-\partial_{y_i}\delta_y)\right)\,,
			$$
			where  $\Delta_h^j(-\partial_{y_i}\delta_y)=-\frac1h(\partial_{y_i}\delta_{y+he_j}-\partial_{y_i}\delta_y)\,.$
			
			Hence, estimate \eqref{sbrigati} and Lagrange's Theorem implies
			\begin{align*}
				&\left|D^l_xR^h_{i,j}(x,y)-D^l_xR^k_{i,j}(x,y)\right|\le C\norm{\Delta_h^j(-\partial_{y_i}\delta_y)-\Delta_k^j(-\partial_{y_i}\delta_y)}_{-(2+\alpha)}\\&=\sup\limits_{\norm{\phi}_{2+\alpha}\le 1}\left(\frac{\partial_{y_i}\phi(y+he_j)-\partial_{y_i}\phi(y)}{h}-\frac{\partial_{y_i}\phi(y+ke_j)-\partial_{y_i}\phi(y)}{k}\right)\\&=\sup\limits_{\norm{\phi}_{2+\alpha}\le 1}\left(\partial^2_{y_iy_j}\phi(y_{\phi,h})-\partial^2_{y_iy_j}\phi(y_{\phi,k})\right)\le\sup\limits_{\norm{\phi}_{2+\alpha}\le 1}|y_{\phi,h}-y_{\phi,k}|^\alpha\le |h|^\alpha+|k|^\alpha\,,
			\end{align*}
			for a certain $y_{\phi,h}$ in the line segment between $y$ and $y+he_j$ and $y_{\phi,k}$ in the line segment between $y$ and $y+ke_j$.
			
			Since the last term goes to $0$ when $h,k\to0$, we have proved that the incremental ratio \eqref{alotteriarubabbeh} and its first and second derivative w.r.t $x$ are Cauchy sequences in $h$, and so converging when $h\to0$. This proves that $D^l_x\dm{U}$ is twice differentiable with respect to $y$, for all $0\le|l|\le2\,$.
			
			In order to show the H\"{o}lder bound for $\dm{U}$ w.r.t. $y$, we consider $y,y'\in\Omega$ and we consider the function
			$$
			R^h_{i,j}(x,y)-R^h_{i,j}(x,y')\,.
			$$
			Then we know from the linearity of \eqref{linDuDm}
			$$
			R^h_{i,j}(x,y)-R^h_{i,j}(x,y')=v(t_0,x;\Delta^j_h(-\partial_{y_i}\delta_y)-\Delta^j_h(-\partial_{y_i}\delta_{y'}))\,,
			$$
			and so, using \eqref{sbrigati} and 
			$$
			\norm{R^h_{i,j}(\cdot,y)-R^h_{i,j}(\cdot,y')}_{2+\alpha}\le C\norm{\Delta^j_h(-\partial_{y_i}\delta_y)-\Delta^j_h(-\partial_{y_i}\delta_{y'})}_{-(2+\alpha)}\,.
			$$
			
			Now we pass to the limit when $h\to0$. 
			It is immediate to prove that
			$$
			\Delta^j_h(-\partial_{y_i}\delta_y)-\Delta^j_h(-\partial_{y_i}\delta_{y'})\overset{h\to0}{\longrightarrow}\partial_{y_j}\partial_{y_i}\delta_y-\partial_{y_j}\partial_{y_i}\delta_{y'}\qquad\mbox{in }\mathcal{C}^{-(2+\alpha)}\,.
			$$
			
			Since $D^l_xR^h_{i,j}(x,y)\to \partial^2_{y_iy_j}D^l_x\dm{U}(x,y)$ for all $|l|\le2$, we can use Ascoli-Arzel\`a to obtain that
			$$
			\norm{\partial^2_{y_iy_j}\dm{U}(t,\cdot,m,y)-\partial^2_{y_iy_j}\dm{U}(t,\cdot,m,y')}_{2+\alpha}\le C\norm{\partial_{y_j}\partial_{y_i}\delta_y-\partial_{y_j}\partial_{y_i}\delta_{y'}}_{-(2+\alpha)}\le C|y-y'|^\alpha\,,
			$$
			which proves \eqref{lentezza} and concludes the proof.
		\end{proof}
	\end{cor}
	
	We conclude this part with a last property on the derivative $D_mU$, which will be essential in order to prove the uniqueness of solutions for the Master Equation.
	
	\begin{cor}\label{delarue}
		The function $U$ satisfies the following Neumann boundary conditions:
		\begin{equation*}
			\begin{split}
				&a(x)D_x\dm{U}(t,x,m,y)\cdot\nu(x)=0\,,\qquad\forall x\in\partial\Omega, y\in\Omega,t\in[0,T],m\in\mathcal{P}(\Omega)\,,\\
				&a(y)D_mU(t,x,m,y)\,\,\,\,\cdot\nu(y)=0\,,\qquad\forall x\in\Omega, y\in\partial\Omega,t\in[0,T],m\in\mathcal{P}(\Omega)\,.
			\end{split}
		\end{equation*}
		\begin{proof}
			Since $\dm{U}(t_0,x,m_0,y)=v(t_0,x)$, where $(v,\mu)$ is the solution of \eqref{linDuDm} with $\mu_0=\delta_y$, the first condition is immediate because of the Neumann condition of \eqref{linDuDm}.
			
			For the second condition, we consider $y\in\partial\Omega$ and we take
			$$
			\mu_0=-\partial_w(\delta_y)\,,\qquad\mbox{with }w=a(y)\nu(y)\,.
			$$
			We want to prove that $(v,\mu)=(0,\mu)$ is a solution of \eqref{linDuDm} with $\mu_0=-\partial_w\delta_y$, where $\mu$ is the unique solution in the sense of Definition \ref{canzonenuova} of
			$$
			\begin{cases}
				\mu_t-\mathrm{div}(a(x)D\mu)-\mathrm{div}(\mu (H_p(x,Du)+\tilde{b}))=0\,,\\
				\mu(t_0)=\mu_0\,,\\
				\left(a(x)D\mu+\mu (H_p(x,Du)+\tilde{b})\right)\cdot\nu_{|\partial\Omega}=0\,. 
			\end{cases}
			$$
			
			We only have to check that, if $\mu$ is a solution of this equation, then $v=0$ solves
			\begin{equation}\label{muovt}
				\begin{cases}
					-v_t-\mathrm{tr}(a(x)D^2v)+H_p(x,Du)\cdot Dv=\mathlarger{\frac{\delta F}{\delta m}}(x,m(t))(\mu(t))\,,\\
					v(T,x)=\mathlarger{\frac{\delta G}{\delta m}}(x,m(T))(\mu(T))\,,\\
					a(x)Dv\cdot\nu_{|\partial\Omega}=0\,,
				\end{cases}
			\end{equation}
			
			which reduces to prove that
			$$
			\dm{F}(x,m(t))(\mu(t))=\dm{G}(x,m(T))(\mu(T))=0\,.
			$$
			We will give a direct proof.
			
			Choosing a test function $\phi(t,y)$ satisfying \eqref{hjbfp}, with $\psi(t,y)=0$ and $\xi(y)=\dm{F}(x,m(t),y)$, we have from boundary conditions of $\dm{F}$ that $\phi$ is a $\mathcal{C}^{\frac{1+\alpha}{2},1+\alpha}$ function satisfying Neumann boundary conditions.
			
			It follows from the weak formulation of $\mu$ that
			$$
			\dm{F}(x,m(t))(\mu(t))=\langle \mu(t),\dm{F}(x,m(t),\cdot)\rangle=\langle\mu_0,\phi(0,\cdot)\rangle=0\,,
			$$
			since $aD\phi\cdot\nu_{|\partial\Omega}=0$ and
			$$
			\langle\mu_0,\phi(0,\cdot)\rangle=\langle -\partial_w\delta_y,\phi(0,\cdot)\rangle=a(y)D\phi(0,y)\cdot\nu(y)=0\,.
			$$
			Same computations hold for $\dm{G}$, proving that $v=0$ satisfies \eqref{muovt}.
			
			Then we can easily conclude:
			\begin{align*}
				a(y)D_mU(t_0,x,m_0,y)\cdot\nu(y)&=D_y\dm{U}(t_0,x,m_0,y)\cdot w\\&=\left\langle\dm{U}(t_0,x,m_0,\cdot),\mu_0\right\rangle=v(t_0,x)=0\,.
			\end{align*}
		\end{proof}
	\end{cor}
	
	\section{Solvability of the first-order Master Equation}
	
	The $\mathcal{C}^1$ character of $U$ with respect to $m$ is crucial in order to prove the main theorem of this chapter.

	\begin{proof}[Proof of Theorem \ref{settepuntouno}]
		We start from the existence part.\\
		
		\emph{Existence}. We start assuming that $m_0$ is a smooth and positive function satisfying \eqref{neumannmzero}, and we consider $(u,m)$ the solution of $MFG$ system starting from $m_0$ at time $t_0$. Then
		$$
		\partial_t U(t_0,x,m_0)
		$$
		can be computed as the sum of the two limits:
		$$
		\lim\limits_{h\to0} \frac{U(t_0+h,x,m_0)-U(t_0+h,x,m(t_0+h))}{h}
		$$
		and
		$$
		\lim\limits_{h\to0} \frac{U(t_0+h,x,m(t_0+h))-U(t_0,x,m_0)}{h}\,.
		$$
		The second limit, using the very definition of $U$, is equal to
		\begin{align*}
			\lim\limits_{h\to0} \frac{u(t_0+h,x)-u(t_0,x)}{h}=u_t(t_0,x)=-\mathrm{tr}(a(x)D^2u(t_0,x))+H(x,Du(t_0,x))\\-F(x,m(t_0))=-\mathrm{tr}(a(x)D^2_xU(t_0,x,m_0))+H(x,D_xU(t_0,x,m_0))-F(x,m_0)\,.
		\end{align*}
		As regards the first limit, defining $m_s:=(1-s)m(t_0)+sm(t_0+h)$ and using the $\mathcal{C}^1$ regularity of $U$ with respect to $m$, we can write it as
		\begin{align*}
			-\lim\limits_{h\to 0}\int_0^1\into\dm{U}(t_0+h,x,m_s,y)\frac{(m(t_0+h,y)-m(t_0,y))}h\,dyds\\
			=-\int_0^1\into\dm{U}(t_0,x,m_0,y)m_t(t_0,y)\,dyds=\into\dm{U}(t_0,x,m_0,y)m_t(t_0,y)\,dy\\=-\into\dm{U}(t_0,x,m_0,y) \,\mathrm{div}\!\left(a(y)Dm(t_0,y) +m(t_0,y)(\tilde{b}+H_p(y,Du(t_0,y)))\right)dy\,.
		\end{align*}
		Taking into account the representation formula \eqref{reprform} for $\dm{U}$ , we integrate by parts and use the boundary condition of $\dm{U}$ and $m$ to obtain
		\begin{align*}
			\into\left[H_p(y,D_xU(t_0,y,m_0))D_mU(t_0,x,m_0,y)-\mathrm{tr}\left(a(y)D_yD_mU(t_0,x,m_0,y)\right)\right]dm_0(y)
		\end{align*}
		So with the computation of the two limits we obtain
		\begin{align*}
			&\partial_t U(t,x,m)=-\mathrm{tr}\left(a(x)D_x^2 U(t,x,m)\right)+H\left(x,D_x U(t,x,m)\right)\\&-\mathlarger{\into}\mathrm{tr}\left(a(y)D_y D_m U(t,x,m,y)\right)dm(y)+\\&\mathlarger{\into} D_m U(t,x,m,y)\cdot H_p(y,D_x U(t,y,m))dm(y)- F(x,m)\,.
		\end{align*}
		So the equation is satisfied for all $m_0\in\mathcal{C}^\infty$ satisfying \eqref{neumannmzero}, and so, with a density argument, for all $m_0\in\mathcal{P}(\Omega)$.
		
		The boundary conditions are easily verified thanks to Corollary \ref{delarue}. This concludes the existence part.\\
		
		\emph{Uniqueness}. Let $V$ be another solution of the Master Equation \eqref{Master} with Neumann boundary conditions. We consider, for fixed $t_0$ and $m_0$, with $m_0$ smooth satisfying \eqref{neumannmzero}, the solution $\tilde{m}$ of the Fokker-Planck equation:
		\begin{equation*}
			\begin{cases}
				\tilde{m}_t-\mathrm{div}(a(x)D\tilde{m})-\mathrm{div}\left(\tilde{m}\left(H_p(x,D_xV(t,x,\tilde{m}))+\tilde{b}\right)\right)=0\,,\\
				\tilde{m}(t_0)=m_0\,,\\
				\left[a(x)D\tilde{m}+(\tilde{b}+D_xV(t,x,\tilde{m}))\right]\cdot\nu(x)_{|\partial\Omega}=0\,.
			\end{cases}
		\end{equation*}
		This solution is well defined since $D_xV$ is Lipschitz continuous with respect to the measure variable.
		
		Then we define $\tilde{u}(t,x)=V(t,x,\tilde{m}(t))$. Using the equations of $V$ and $\tilde{m}$, we obtain
		\begin{align*}
			\tilde{u}_t(t,x)=&\,V_t(t,x,\tilde{m}(t))+\into\dm{V}(t,x,\tilde{m}(t),y)\,\tilde{m}_t(t,y)\,dy\\
			=&\,V_t(t,x,\tilde{m}(t))+\into\dm{V}(t,x,\tilde{m}(t),y)\,\mathrm{div}\!\left(a(y)D\tilde{m}(t,y)\right)\,dy\\+&\,\into \dm{V}(t,x,\tilde{m}(t),y)\,\mathrm{div}\!\left(\tilde{m}\left(H_p(x,D_xV(t,x,\tilde{m}))+\tilde{b}\right)\right)\,dy\,.
		\end{align*}
		
		We compute the two integrals by parts. As regards the first, we have
		\begin{align*}
			&\into\dm{V}(t,x,\tilde{m}(t),y)\,\mathrm{div}\!\left(a(y)D\tilde{m}(t,y)\right)\,dy\\=-&\into a(y)D\tilde{m}(t,y)\,D_mV(t,x,\tilde{m}(t),y)\,dy+\int_{\partial\Omega}\dm{V}(t,x,\tilde{m}(t),y)\,a(y)D\tilde{m}(t,y)\cdot\nu(t,y)\,dy\\
			=&\into\mathrm{div}(a(y)D_yD_mV(t,x,\tilde{m}(t),y))\,\tilde{m}(t,y)\,dy-\into a(y)D_mV(t,x,\tilde{m}(t),y)\cdot\nu(y)\tilde{m}(t,y)dy\\+&\int_{\partial\Omega}\dm{V}(t,x,\tilde{m}(t),y)\,a(y)D\tilde{m}(t,y)\cdot\nu(t,y)\,dy\,,
		\end{align*}
		
		while for the second
		\begin{align*}
			&\into\dm{V}(t,x,\tilde{m}(t),y)\,\mathrm{div}\!\left(\tilde{m}\left(H_p(x,D_xV(t,x,\tilde{m}))+\tilde{b}\right)\right)\,dy\\
			=-&\into\left(H_p(x,D_xV(t,x,\tilde{m}))+\tilde{b}\right)D_mV(t,x,\tilde{m},y)\tilde{m}(t,y)dy\\+&\int_{\partial\Omega}\dm{V}(t,x,\tilde{m}(t),y)\left(H_p(x,D_xV(t,x,\tilde{m}))+\tilde{b}\right)\cdot\nu(y)\,\tilde{m}(t,y)dy\,.
		\end{align*}
		
		Putting together these estimates and taking into account the boundary conditions on $V$ and $m$:
		$$
		\left[a(x)D\tilde{m}+(\tilde{b}+D_xV(t,x,\tilde{m}))\right]\cdot\nu(x)_{|x\in\partial\Omega}=0\,,\qquad a(y)D_mV(t,x,m,y)\cdot\nu(y)_{|y\in\partial\Omega}=0\,,
		$$
		and the relation between the divergence and the trace term
		$$
		\mathrm{div}(a(x)D\phi(x))=\mathrm{tr}(a(x)D^2\phi(x))+\tilde{b}(x)D\phi(x)\,,\qquad\forall\phi\in W^{2,\infty}(\Omega)\,,
		$$
		we find
		\begin{align*}
			\tilde{u}_t(t,x)=&\,V_t(t,x,\tilde{m}(t))+\into\mathrm{tr}(a(y)D_yD_mV(t,x,\tilde{m},y))\,d\tilde{m}(y)\\-&\,\into -H_p(y,D_xV(t,y,\tilde{m})) D_mV(t,x,\tilde{m},y)\,d\tilde{m}(y)\\=&\,-\mathrm{tr}(a(x)D^2_xV(t,x,\tilde{m}(t)))+H(x,D_xV(t,x,\tilde{m}(t)))-F(x,\tilde{m}(t))\\=&\,-\mathrm{tr}(a(x)D^2\tilde{u}(t,x))+H(x,D\tilde{u}(t,x))-F(x,\tilde{m}(t))\,.
		\end{align*}
		This means that $(\tilde{u},\tilde{m})$ is a solution of the MFG system \eqref{meanfieldgames}-\eqref{fame}. Since the solution of the Mean Field Games system is unique, we get $(\tilde{u},\tilde{m})=(u,m)$ and so $V(t_0,x,m_0)=U(t_0,x,m_0)$ whenever $m_0$ is smooth.\\
		Then, using a density argument, the uniqueness is proved.
	\end{proof}
	
	\vspace{1.5cm}
	\textbf{Acknowledgements.} 
	I wish to sincerely thank P. Cardaliaguet and A. Porretta for the help and the support during the preparation of this article. I wish to thank also F. Delarue for the enlightening ideas he gave to me. 
	
\end{document}